\documentclass[10pt,a4paper,reqno]{amsart}

\usepackage[colorlinks=true, linkcolor=magenta, citecolor=cyan, urlcolor=blue,hyperfootnotes=false]{hyperref}
\usepackage{amssymb,graphicx,mathrsfs,xcolor,enumitem}
\numberwithin{equation}{section}\newtheorem{theorem}{Theorem}[section]
\newtheorem{corollary}[theorem]{Corollary}\newtheorem{lemma}[theorem]{Lemma}

\newtheorem{proposition}[theorem]{Proposition}\theoremstyle{remark}
\newtheorem{remark}{Remark}[section]
\theoremstyle{definition}

\newcommand{\bra}[1]{\langle #1 \rangle}
\newcommand{\one}[1]{\mathbf{1}_{#1}}
\newcommand{\sgn}{\mathop{\textrm{sgn}}}

\title[Helmholtz equation]
{Helmholtz and dispersive equations with variable coefficients
on exterior domains}
\date{\today}    
\author{Federico Cacciafesta}
\address{Federico Cacciafesta: 
SAPIENZA --- Universit\`a di Roma,
Dipartimento di Matematica, 
Piazzale A.~Moro 2, I-00185 Roma, Italy}
\email{cacciafe@mat.uniroma1.it}
\author{Piero D'Ancona}
\address{Piero D'Ancona: 
SAPIENZA --- Universit\`a di Roma,
Dipartimento di Matematica, 
Piazzale A.~Moro 2, I-00185 Roma, Italy}
\email{dancona@mat.uniroma1.it}
\author{Renato Luc\`a}
\address{Renato Luc\`a: 
Instituto de Ciencias Matematicas, 
Consejo de Investigaciones Cientificas, 
C. Nicolas Cabrera 13-15, 28049 Madrid, Spain.
Supported by the ERC grant 277778 and MINECO grant SEV-2011-0087 (Spain)}
\email{renato.luca@icmat.es}

\thanks{The authors were partially supported by the
Italian Project FIRB 2012 ``Dispersive
dynamics: Fourier Analysis and Variational Methods''}
\begin{document}
\begin{abstract}
  We prove smoothing estimates in Morrey-Campanato spaces
  for a Helmholtz equation
  \begin{equation*}
    -Lu+zu=f,
    \qquad
    -Lu:=\nabla^{b}(a(x)\nabla^{b}u)-c(x)u,
    \qquad
    \nabla^{b}:=\nabla+ib(x)
  \end{equation*}
  with fully variable coefficients, of limited regularity, 
  defined on the exterior of a starshaped compact obstacle
  in $\mathbb{R}^{n}$, $n\ge3$, with Dirichlet boundary conditions.
  The principal part of the operator is a long range perturbation
  of a constant coefficient operator, while the lower order terms 
  have an almost critical decay. We give explicit conditions
  on the size of the perturbation which prevent trapping.

  As an application, we prove smoothing
  estimates for the Schr\"{o}dinger flow $e^{itL}$
  and the wave flow $e^{it \sqrt{L}}$
  with variable coefficients on exterior domains
  and Dirichlet boundary conditions.
\end{abstract}\maketitle


\section{Introduction}\label{sec:introduction}

The Helmholtz equation
\begin{equation}\label{eq:baseh}
  \Delta u+z u=f(x)
\end{equation}
where $z\in \mathbb{C}$ and $x$ varies on $\mathbb{R}^{n}$
or on the exterior $\Omega=\mathbb{R}^{n}\setminus K$ 
of an obstacle $K$, is used to model standing waves 
in many different applications in physics and engineering. 
When $z\not\in \mathbb{R}$, \eqref{eq:baseh} can be written as
the resolvent equation $u=R_{0}(z)f$ for
$R_{0}(z)=(-\Delta-z)^{-1}$, 
and the interesting problem is to prove
uniform estimates for $z$ approaching the spectrum
of the operator.
In addition,
smoothing estimates for \eqref{eq:baseh} have become an
important tool to prove smoothing and Strichartz estimates
for many dispersive equations including
wave, Schr\"{o}dinger, Klein-Gordon and Dirac equations.

The theory of \eqref{eq:baseh} is classical, while in recent
years generalizations of \eqref{eq:baseh} with 
lower order terms have attracted some attention. In
\cite{PerthameVega99-a},
\cite{PerthameVega08-a}
the Helmholtz equation on $\mathbb{R}^{n}$
with a variable refraction index
\begin{equation*}
  \Delta u-n(x)u+z u=f(x)
\end{equation*}
is studied; first order perturbations
were examined in \cite{Fanelli09-b} on the whole space,
and on exterior domains in \cite{BarceloFanelliRuiz13-a}.

Resolvent estimates for elliptic operators of various type
have been studied intensely, initially in the framework
of spectral theory and scattering, more
recently for use in proving decay estimates for
evolution equations. The two strains of research
have independent origin in the seminal papers
\cite{Kato65-a} (see also \cite{KatoYajima89-a}),
\cite{Agmon75-a} and
in Morawetz' work on wave equations
\cite{Morawetz68-a}, \cite{Morawetz75-a};
Fourier transform with respect to the time variable
acts as a bridge between the two points of view.
Resolvent estimates for $-\Delta+V(x)$
were proved in 
\cite{Ben-Artzi92-a},
\cite{Ben-ArtziKlainerman92-a},
and more recently for singular potentials in
\cite{BarceloRuizVega09-a}.
Magnetic potentials $(i \nabla-A(x))^{2}+V(x)$ 
were studied in
\cite{DAnconaFanelli08-a} (small $A$)
and
\cite{ErdoganGoldbergSchlag09-a},
\cite{CardosoCuevasVodev13-a}
(large $A$).

The case of
Laplace-Beltrami operators on some classes of manifolds
has also attracted intense attention. 
Among the many contributions, we mention
\cite{Burq98-b},
\cite{VasyZworski00-a},
\cite{Burq02-a},
\cite{CardosoVodev02-a},
\cite{Vodev04-a},
\cite{BoucletTzvetkov08-a}
for high frequency resolvent estimates on 
asymptotically flat or conical manifolds, possibly
with boundary. Global smoothing estimates for all
frequencies in the case of manifolds which are
flat outside a compact region
were obtained in 
\cite{RodnianskiTao07-a}, and
quantitative bounds for more general manifolds
were proved in
\cite{RodnianskiTao11-a}.

A key point in the previous results is that the principal
part of the operator, and the obstacle, must be \emph{nontrapping}, meaning that the Hamiltonian flow
of the operator
leaves compact sets in a finite time. Actually, weaker
resolvent estimates are still true in the trapped case,
but the bounds may grow exponentially in the frequency,
making them unsuitable for applications to dispersive
equations.
The importance of this condition in the context of
smoothing estimates was noted at an early stage, see
\cite{CraigKappelerStrauss95-a},
\cite{Doi00-a},
\cite{Burq02-a}.

In this paper we prove smoothing estimates on 
a starshaped exterior domain
for a Helmholtz equation  with fully variable coefficients,
and for the corresponding Schr\"{o}dinger and wave equations.
In summary,
under our assumptions, the principal part is a (small)
long range perturbation of a constant coefficient operator,
while the lower order terms have an almost critical decay
at infinity. Our method of proof, based
on an adaptation of the multiplier method, permits to give
explicit quantitative conditions on the coefficients,
which may be regarded as
a form of \emph{quantitative nontrapping conditions}.
By this we mean that the assumptions can be explicitly checked
in concrete examples, in contrast with the usual formulation
of nontrapping in terms of the bicharacteristic flow.

We consider a Helmholtz equation 
\begin{equation}\label{eq:helm}
  \nabla^b\cdot(a(x)\nabla^bv)-c(x)v
    +(\lambda+i\varepsilon) v=f,\qquad
  \lambda\in \mathbb{R},\quad \epsilon>0
\end{equation}
on an exterior domain $\Omega=\mathbb{R}^{n}\setminus \omega$,
with $\omega$ bounded and possibly empty;
coefficients are real valued, $a(x)=[a_{jk}(x)]_{j,k=1}^{n}$
and $\nabla^{b}$ denotes the covariant derivatives
\begin{equation*}
  \nabla^{b}=(\partial_{1}^{b},\dots,\partial_{n}^{b}),\qquad
  \partial^{b}_{j}=\partial_{j}+ib_{j}(x).
\end{equation*}
The conditions on the coefficients are the following:

\textsc{Assumptions on $a(x)$}.
The symmetric, real valued matrix
$a(x)=[a_{jk}(x)]_{j,k=1}^{n}$ satisfies
\begin{equation}\label{eq:assaabove}
  NI\ge a(x)\ge\nu I,
  \qquad
  N\ge\nu>0,
\end{equation}
and denoting with $|a(x)|$ the operator norm of the
matrix $a(x)$ and writing
$|a'|=\sum_{|\alpha|=1}|\partial^{\alpha}a(x)|$,
$|a''|=\sum_{|\alpha|=2}|\partial^{\alpha}a(x)|$ and
$|a'''|=\sum_{|\alpha|=3}|\partial^{\alpha}a(x)|$,
we have
\begin{equation}\label{eq:assader}
  |a'(x)|+|x||a''(x)|+|x|^{2}|a'''(x)|\le C_{a} 
  \langle x\rangle^{-1-\delta},
  \qquad
  \delta\in(0,1).
\end{equation}

\textsc{Assumptions on $b(x)$}.
The coefficients $b(x)=(b_{1},\dots,b_{n})$ are real valued
and the matrix 
$db(x):=
  [\partial_{j}b_{\ell}-\partial_{\ell}b_{j}]_{j,\ell=1}^{n}$
satisfies
\begin{equation}\label{eq:assdb}
  \textstyle
  |db(x)|\le \frac{C_{b}}{|x|^{2+\delta}+|x|^{2-\delta}}
\end{equation}

\textsc{Assumptions on $c(x)$}.
The potential $c(x)$ is real valued and satisfies
\begin{equation}\label{eq:assc}
  \textstyle
  -\frac{C_{-}^{2}}{|x|^{2+\delta}+|x|^{2-\delta}}
  \le
  c(x)
  \le
  \frac{C_{+}^{2}}{|x|^{2}}.
\end{equation}
Moreover we assume that $c(x)$ is \emph{repulsive} with respect
to the metric $a(x)$, meaning that
\begin{equation}\label{eq:assac}
  \textstyle
  a(x)x \cdot\nabla c\le 
  \frac{C_{c}}{|x|\langle x\rangle^{1+\delta}}
\end{equation}

\textsc{Assumptions on the domain}.
The domain $\Omega \subseteq\mathbb{R}^{n}$ is an 
exterior domain, i.e.
the complement of a compact and possibly empty set.
We assume that $\partial \Omega$ is $C^{1}$
and $a(x)$--\emph{starshaped},
meaning that at all points of $\partial \Omega$
the exterior normal $\vec\nu$ to $\partial\Omega$ satisfies
\begin{equation}\label{eq:assbdry}
  a(x)x\cdot\vec{\nu}\le0.
\end{equation}
When $a=I$,
\eqref{eq:assbdry} reduces to the condition that
$\omega$ is starshaped with respect to the origin.

\begin{remark}[Selfadjointness]
  \label{rem:selfadjoint}
  By the previous assumptions, the operator
  $L=-A^{b}+c$ is symmetric and
  satisfies the inequality
  \begin{equation*}
    (Lv,v)_{L^{2}(\Omega)}\ge
    \nu\|\nabla^{b}v\|_{L^{2}}-
      C_{-}^{2}
      \|(x|^{2+\delta}+|x|^{2-\delta})^{-1/2}v\|_{L^{2}(\Omega)}
    \qquad
    \forall v\in C_{c}(\Omega).
  \end{equation*}
  By the magnetic Hardy inequality
  \eqref{eq:hardymag1}, this implies that
  \begin{equation*}
    (Lv,v)_{L^{2}(\Omega)}\ge-C\|u\|_{L^{2}(\Omega)}^{2}.
  \end{equation*}
  Thus the operator $L$ has a
  selfadjoint Friedrichs extension on $L^{2}(\Omega)$,
  which is sufficient for our purposes.
  Actually, it possible to prove that $L$
  is essentially selfadjoint and hence the selfadjoint
  extension is unique: this follows by Chernoff's result
  \cite{Chernoff73-a}
  since the wave equation
  $u_{tt}+Lu=0$ has finite speed of propagation.
\end{remark}

Our first result is a \emph{homogeneous} 
smoothing resolvent estimate,
expressed in terms of the Morrey-Campanato type norms
\begin{equation*}
  \textstyle
  \|v\|_{\dot X}^{2}:=
  \sup\limits_{R>0}\frac{1}{R^{2}}\int_{\Omega\cap\{|x|=R\}}
  |v|^{2}dS,
  \qquad
  \|v\|_{\dot Y}^{2}:=
  \sup\limits_{R>0}\frac{1}{R}\int_{\Omega\cap\{|x|\le R\}}
  |v|^{2}dx,
\end{equation*}
while the $\dot Y^{*}$ norm is predual to the $\dot Y$ norm
(see an explicit characterization in Section \ref{sec:norms}).
Note that the result is only partially
satisfactory in the case $n=3$ which will be
considered in detail below (see Remark \ref{rem:suffcaseA}).

\begin{theorem}\label{the:1}
  Let $n\ge3$.
  Consider the Helmholtz equation \eqref{eq:helm}
  with Dirichlet boundary 
  conditions on an exterior domain $\Omega$
  as in \eqref{eq:assbdry}. Assume that
  \eqref{eq:assaabove}, \eqref{eq:assader},
  \eqref{eq:assdb},
  \eqref{eq:assc} and \eqref{eq:assac} hold. Finally, assume
  the ratio $N/\nu$ satisfies
  \begin{equation}\label{eq:assratio}
    \textstyle
    \frac{N}{\nu}\le
    \sqrt{\frac{n^{2}+2n+15}{6(n+2)}}
    \quad \text{for}\ 
    3\le n\le 46,
    \qquad \qquad
    \frac{N}{\nu}<
    \frac{3n-1}{n+3}
    \quad \text{for}\ 
    n\ge 47
  \end{equation}
  and the constants $C_{a},C_{-},C_{c},C_{b}$ are small enough that
  \begin{equation}\label{eq:smallC}
    \textstyle
    C_{a}(N+C_{a})\le \frac{K \delta}{24n},\quad
    C_{b}\le \frac{K \delta}{5N^{2}},\quad
    C_{-}\le \frac{K \delta}{18N(N+2)},\quad
    C_{c}\le K \delta,
  \end{equation}
  where 
  $K=\min\left\{1,\frac{\nu^{2}}{9},
  \frac{(n+3)\nu^{2}}{18}(\frac{3n-1}{n+3}-\frac N\nu)\right\}$.
  Then the solution to \eqref{eq:helm} satisfies
  \begin{equation}\label{eq:thesisA}
    \|v\|_{\dot X}^{2}
    +\|\nabla^{b}v\|_{\dot Y}^{2}
    \le
    M_{0}
    \|f\|_{\dot Y^{*}}^{2}
  \end{equation}
  and
  \begin{equation}\label{eq:thesisA2}
    \textstyle
    |\lambda|\|v\|_{\dot Y}^{2}
    \le
    2n^{2}(\nu+1)^{2}M_{0}
    \|f\|_{\dot Y^{*}}^{2},
    \qquad
    |\epsilon|\|v\|_{\dot Y}^{2}
    \le
    9(\nu+1)M_{0}
    \|f\|_{\dot Y^{*}}^{2}
  \end{equation}
  where $M_{0}=64n^{2}K^{-2}(\nu+1)^{2} (C_{+}+\nu+1)^{2}$.
\end{theorem}

\begin{remark}[]\label{rem:nontrapping}
  The use of multiplier methods has several advantages
  versus the phase space approach, besides simplicity.
  Indeed, one can prove sharp estimates in terms
  of Morrey-Campanato norms which are stronger than
  the usual weighted $L^{2}$ norms; in addition,
  one obtains quantitative bounds
  which would be impossible to prove when using e.g.~Fredholm
  theory. The technique used here was introduced in
  \cite{PerthameVega99-a},
  and then improved in
  \cite{BarceloRuizVega06-a},
  \cite{BarceloRuizVega09-a},
  \cite{DAnconaRacke12-a},
  \cite{BarceloFanelliRuiz13-a}, and
  the main novelty of the present paper is the
  adaptation of the method to a general elliptic
  operator with variable coefficients.
  In spirit, this paper is close to
  \cite{RodnianskiTao11-a}
  where a version of the multiplier method for manifolds
  is developed.

  However, in our opinion, the best feature of the method is
  the possibility to obtain explicit
  (although non sharp) criteria
  to check if an operator is nontrapping. 
  In addition, tracking the constants with precision allows to
  see their qualitative dependence on the parameters
  of the problem.
  As a simple example, consider the diagonal case
  \begin{equation*}
    \nabla \cdot(\alpha(x)\nabla u)+\frac{C}{|x|^{2}}u+zu=f
    \qquad\text{on}\quad \mathbb{R}^{4},
  \end{equation*}
  where $C\ge0$ and
  $\alpha(x):\mathbb{R}^{4}\to \mathbb{R}$ is a
  scalar function satisfying for some $\delta\in(0,1)$
  \begin{equation*}
    \textstyle
    1\le \alpha(x)\le \sqrt{\frac{13}{12}},\qquad
    |\alpha'|+|x||\alpha''|+|x|^{2}|\alpha'''|\le
    \frac{\delta}{900}\bra{x}^{-1-\delta}.
  \end{equation*}
  Under these conditions, all the assumptions of the
  Theorem are satisfied, and
  the smoothing estimates \eqref{eq:thesisA}, 
  \eqref{eq:thesisA2} are true.
\end{remark}

\begin{remark}[]\label{rem:potential}
  For high energies $\lambda>>1$
  the estimates are valid for the more general class 
  of potentials with Coulomb decay $|c(x)|\le \frac{C_{0}}{|x|}$.
  Indeed, if we perturb the equation with a term $c_{0}(x)v$:
  \begin{equation*}
    A^{b}v-c(x)v+(\lambda+i \epsilon)v=f-c_{0}(x)v=:\widetilde{f}(x)
  \end{equation*}
  and we apply the previous result, we obtain
  \begin{equation*}
    \|\nabla v\|_{\dot Y}
    +\sqrt{|\lambda|+|\epsilon|}\|v\|_{\dot Y}+\|v\|_{\dot X}
    -C\|c_{0}(x)v\|_{\dot Y^{*}}
    \le C\|f\|_{\dot Y^{*}}.
  \end{equation*}
  Since 
  \begin{equation*}
    \textstyle
    \|c_{0}v\|_{\dot Y^{*}}
    \lesssim
    \||x|c_{0}\|_{L^{\infty}}
    \|v\|_{\dot Y}
  \end{equation*}
  we see that we can absorb the negative term in the
  term $|\lambda|^{\frac12}\|v\|_{\dot Y}$, provided
  \begin{equation*}
    \textstyle
    \lambda^{\frac12} \gtrsim 
    \||x|c_{0}\|_{L^{\infty}}.
  \end{equation*}
  Note that a similar argument allows to absorb small
  first order perturbations of critical decay
  in the term $\|\nabla v\|_{\dot Y}$.
\end{remark}

\begin{remark}[]\label{rem:precisecond}
  For the proof of the Theorem, instead of the
  second condition in \eqref{eq:assratio}, it is sufficient
  to assume the following weaker pointwise inequality:
  \begin{equation}\label{eq:asscaseA}
    2|a(x)|_{HS}^{2}+\overline{a}^{2}
      -6 \overline{a}(x)\widehat{a}(x)
      +15 \widehat{a}^{2}(x)
      -12|a(x) \widehat{x}|^{2}\ge0
      \qquad
      \forall x\in \Omega.
  \end{equation}
  where $|a(x)|_{HS}$ is the Hilbert-Schmidt norm of the
  matrix $a(x)$,
  $\overline{a}(x)$ its trace,
  $\widehat{a}(x)=a(x)\widehat{x} \cdot \widehat{x}$
  and $\widehat{x}=x/|x|$. Since
  \begin{equation*}%
    2|a|_{HS}+\overline{a}^{2}
    -6 \overline{a}\widehat{a}+15 \widehat{a}^{2}
    -12|a \widehat{x}|^{2}\ge
    (2n+n^{2}+15)\nu^{2} -6 (n+2)N^{2}
  \end{equation*}
  we see that the second condition in \eqref{eq:assratio}
  implies \eqref{eq:asscaseA}.
\end{remark}

\begin{remark}\label{rem:suffcaseA}
  In the case $n=3$, assumption \eqref{eq:assratio}
  forces $\nu=N$ so that
  $a(x)$ must be a diagonal operator, and this is of
  course too restrictive for our purposes.
  There is not much to gain if we revert to the weaker
  assumption \eqref{eq:asscaseA}: for instance, if
  $a(x)=\mathop{\textrm{diag}}[1,1,1+\epsilon]$ and we choose
  $\widehat{x}=(0,0,1)$, the quantity in \eqref{eq:asscaseA}
  is equal to $-8 \epsilon$, thus generic
  small perturbations of $I$ are ruled out also under the
  weaker assumption.

  For this reason we complement
  Theorem \ref{the:1} with an additional result in which
  $a(x)$ is allowed to be any small perturbation of identity.
  The drawback is that we obtain a slightly weaker
  \emph{nonhomogeneous} estimate, which is expressed in
  terms of the norms
  \begin{equation*}
    \textstyle
    \|v\|_{X}^{2}:=
    \sup\limits_{R>0}\frac{1}{\langle R\rangle^{2}}
    \int_{\Omega\cap\{|x|=R\}}
    |v|^{2}dS,
    \qquad
    \|v\|_{Y}^{2}:=
    \sup\limits_{R>1}\frac{1}{R}\int_{\Omega\cap\{|x|\le R\}}
    |v|^{2}dx,
  \end{equation*}
  and the $Y^{*}$ norm, predual to the $Y$ norm.
  Then we have:
 \end{remark}

\begin{theorem}\label{the:2}
  Let $n=3$.
  Consider the Helmholtz equation \eqref{eq:helm}
  with Dirichlet boundary 
  conditions on an exterior domain $\Omega$
  as in \eqref{eq:assbdry}. 
  Assume that
  \eqref{eq:assaabove}, \eqref{eq:assader} hold and
  that $db(x)$ satisfies
  \begin{equation}\label{eq:assdbstrong}
    \textstyle
    |db(x)|\le \frac{C_{b}}{|x|^{2+\delta}+|x|}
  \end{equation}
  for some $\delta\in(0,1)$,
  while the electric potential $c(x)$ satisfies
  \begin{equation}\label{eq:asscstrong}
    \textstyle
    -\frac{C_{-}^{2}}{\bra{x}^{2+\delta}}
    \le
    c(x)
    \le
    \frac{C_{+}^{2}}{|x|^{2}}
  \end{equation}
  and
  \begin{equation}\label{eq:assacstrong}
    \textstyle
    a(x)x \cdot \nabla c\le
    \frac{C_{c}}{\langle x\rangle^{2+\delta}}.
  \end{equation}
  Finally, assume the principal part is close to identity
  in the following sense:
  \begin{equation}\label{eq:assperturb}
    |a(x)- I|\le  C_{I} \langle x\rangle^{-\delta},
    \qquad
    C_{I}<1/100
  \end{equation}
  and the constants $C_{a},C_{-},C_{I},C_{c},C_{b}$ are
  small enough that
  \begin{equation}\label{eq:smallCstrong}
    \textstyle
    C_{a}\le \frac{\delta}{48000},\quad
    C_{I}\le \frac{\delta}{7200},\quad
    C_{b}\le \frac{\delta}{920},\quad
    C_{-}\le \frac{\delta}{5500},\quad
    C_{c}\le \frac{\delta}{1300}.
  \end{equation}
  Then the solution to \eqref{eq:helm} satisfies the estimates
  \begin{equation}\label{eq:thesisB}
    \|v\|_{X}^{2}
    +\|\nabla^{b}v\|_{Y}^{2}
    \le
    10^{9}(C_{+}^{2}+1)
    \|f\|_{Y^{*}}^{2}
  \end{equation}
  and
  \begin{equation}\label{eq:thesisB2}
    |\lambda|\|v\|_{Y}^{2}
    \le
    10^{10}(C_{+}^{2}+1)^{2}
    \|f\|_{Y^{*}}^{2},
    \qquad
    |\epsilon|\|v\|_{Y}^{2}
    \le
    10^{10}
    (C_{+}^{2}+1) 
    \|f\|_{Y^{*}}^{2}.
  \end{equation}
\end{theorem}

We conclude the Introduction with 
some implications of the previous estimates
for equations of Schr\"{o}dinger and wave type connected
to the operator $A^{b}-c$.

\subsection{Applications}\label{sub:applications}

The estimates in
Theorems \ref{the:1} and \ref{the:2} have several
applications. Natural consequences are
a limiting absorption principle for the operator 
$L=-A^{b}+c$,
the absence of embedded or zero eigenvalues and 
resonances, and the existence 
and uniqueness of solutions for the Helmholtz equation
$(-L+\lambda)v=f$
under a Sommerfeld radiation condition at infinity.
We shall study these and related questions in a
forthcoming paper; here we will focus on the applications
to dispersive evolution equations connected to the operator 
$L=-A^{b}+c$. 

Namely, we consider the Schr\"{o}dinger equation
\begin{equation}\label{eq:schro}
  iu_{t}-A^{b}u+c(x)u=F(t,x),
  \qquad
  u(0,x)=u_{0}(x)
\end{equation}
and the wave equation
\begin{equation}\label{eq:wave}
  u_{tt}-A^{b}u+c(x)u=F(t,x),
  \qquad
  u(0,x)=u_{0}(x),
  \quad
  u_{t}(0,x)=u_{1}(x)
\end{equation}
associated with $L$. Smoothing and decay properties of
solutions are best expressed in terms of the corresponding
\emph{Schr\"{o}dinger flow} $e^{itL}$
and
\emph{wave flow} $e^{it \sqrt{L}}$.
Kato's smoothing theory \cite{Kato65-b}
(see also \cite{KatoYajima89-a})
provides a direct
link between resolvent estimates and smoothing estimates for
the Schr\"{o}dinger flow. Actually it is possible to prove
that estimates for the full resolvent,
\emph{supersmoothing} estimates in the terminology of
Kato--Yajima, are equivalent to
\emph{nonhomogeneous} estimates (like
\eqref{eq:smooschrnh1}, \eqref{eq:smooschrnh2} in the following
Corollary), while the \emph{homogeneous} estimates of the form
\eqref{eq:smooschr} are equivalent to weaker estimates for the
imaginary part of the resolvent, which are properly called
\emph{smoothing} estimates. 
A detailed exposition of the theory, together with an extension
which allows to include wave and Klein-Gordon equations
at the same level of generality, can be found in
\cite{DAncona14-a}.
An additional consequence of the resolvent estimates is the
pointwise decay in time of local norms of the solution,
which can be intepreted as a stronger form of the
RAGE theorem of \cite{Ruelle69-a} for these flows.

In order to simplify the exposition, we shall make
two additional hypotheses on the operator $L$
(but see also Remarks \ref{rem:directmorawetxSch} and
\ref{rem:directmorawetxWE}):

\textsc{Positivity of $L$}.
We assume that $L$ is a \emph{positive} operator,
i.e., $(Lv,v)_{L^{2}(\Omega)}\ge0$ for all $v$ in the
domain of $L$. Note that in view of 
the magnetic Hardy inequality \eqref{eq:hardymag1}
proved below and the 
previous assumptions on the coefficients, we have
\begin{equation*}
  \textstyle
  (Lu,u)_{L^{2}(\Omega)}\ge 
  \nu\|\nabla^{b}u\|_{L^{2}(\Omega)}^{2}
    -C_{-}^{2}\||x|^{-1}u\|^{2}_{L^{2}(\Omega)}
  \ge
  (\frac{(n-2)^{2}}{4}\nu-C_{-}^{2})\||x|^{-1}u\|
      ^{2}_{L^{2}(\Omega)}
\end{equation*}
thus in order to have $L\ge0$ (and actually coercive)
it is sufficient to assume that 
\begin{equation}\label{eq:posL}
  C_{-}< \frac{n-2}{2}\sqrt{\nu}.
\end{equation}

\textsc{Weighted elliptic estimate}. We assume that
the operator $L$ satisfies the weighted $L^{2}$
estimates
\begin{equation}\label{eq:ellest}
  \|\bra{x}^{-\frac12-\sigma}L^{\frac12}v\|_{L^{2}(\Omega)}
  \lesssim
  \|\bra{x}^{-\frac12-\sigma'}
      (-\Delta)^{\frac12}v\|_{L^{2}(\Omega)}
\end{equation}
for $\sigma>\sigma'>0$ close to 0. Note that the right hand
side in \eqref{eq:ellest} is equivalent to
$\|\bra{x}^{-\frac12-\sigma'}\nabla v\|_{L^{2}(\Omega)}$
since $\bra{x}^{-\frac12-}$ is an $A_{2}$ weight. 
This assumption is not too restrictive: indeed, this
kind of elliptic estimates is known in several cases
(see e.g.~\cite{RodnianskiTao11-a} for elliptic operators
with smooth coefficients on $\mathbb{R}^{n}$).
More generally, the estimate 
follows directly from, or can be proved by the techniques of
\cite{CacciafestaDAncona12-a} 
for any selfadjoint operator
$L$ whose heat kernel $e^{tL}$ satisfies an upper gaussian
estimate; this covers the case of elliptic operators on
exterior domains whith bounded coefficients
\cite{Ouhabaz05-a} and magnetic Schr\"{o}dinger operators
with singular coefficients on $\mathbb{R}^{n}$
\cite{Kurata00-a}.

Then we can prove the following results:

\begin{corollary}[]\label{cor:smooschr}
  Assume the selfadjoint operator $L=-A^{b}+c$ with Dirichlet 
  b.c.~on $L^{2}(\Omega)$ and the exterior domain $\Omega$
  satisfy the assumptions of Theorem \ref{the:1}.
  Assume in addition \eqref{eq:posL} and \eqref{eq:ellest}.
  Then we have the estimates
  \begin{equation}\label{eq:smooschr}
    \|\bra{x}^{-1-}e^{itL}f\|
          _{L^{2}_{t}L^{2}_{x}(\Omega)}+
    \|\bra{x}^{-\frac12-}L^{\frac14}e^{itL}f\|
          _{L^{2}_{t}L^{2}_{x}(\Omega)}
    \lesssim\|f\|_{L^{2}(\Omega)},
  \end{equation}
  \begin{equation}\label{eq:smooschrnh1}
    \textstyle
    \|\int_{0}^{t}\bra{x}^{-1-}e^{i(t-s)L}Fds\|
          _{L^{2}_{t}L^{2}_{x}(\Omega)}
    \lesssim
    \|\bra{x}^{1+}F\|
          _{L^{2}_{t}L^{2}_{x}(\Omega)},
  \end{equation}
  \begin{equation}\label{eq:smooschrnh2}
    \textstyle
    \|\int_{0}^{t}\bra{x}^{-\frac12-}
          L^{\frac12}e^{i(t-s)L}Fds\|
          _{L^{2}_{t}L^{2}_{x}(\Omega)}
    \lesssim
    \|\bra{x}^{\frac12+} F\|
          _{L^{2}_{t}L^{2}_{x}(\Omega)}i.
  \end{equation}
  Moreover, for every $f\in L^{2}(\Omega)$,
  we have the RAGE type property
  \begin{equation}\label{eq:ragesch}
    \|\bra{x}^{-\frac12-}e^{itL}f\|_{L^{2}(\Omega)}\to0
    \quad\text{as}\quad 
    t\to\pm \infty.
  \end{equation}
\end{corollary}

\begin{remark}[]\label{rem:directmorawetxSch}
  Note that we can also apply the Morawetz multiplier method
  directly to equation \eqref{eq:schro}, instead of
  using Kato's theory. This approach does not require
  assumptions \eqref{eq:posL}, \eqref{eq:ellest} and
  gives an estimate of the form
  \begin{equation*}
    \|\nabla e^{itL}f\|_{\dot Y L^{2}_{t}}
      +\|e^{itL}f\|_{\dot X L^{2}_{t}}
    \lesssim
    \|f\|_{\dot H^{\frac12}},
  \end{equation*}
  but it does not seem easy to transfer the $1/2$ derivative
  from the right to the left hand side.
  Note also that the norms at the l.h.s.~are of reversed type,
  with an inner integration in $t$.
  However, this method can be used to prove
  \emph{interaction Morawetz} estimates, which will be the
  topic of a forthcoming paper.
\end{remark}

\begin{remark}[]\label{rem:strich}
  Assume, in addition to the above, that the operator
  $L$ coincides with the Laplacian $-\Delta$ outside a bounded
  set, and moreover that the local in time 
  (non endpoint) Strichartz estimates hold for solutions
  of equation \eqref{eq:schro} which are compactly supported
  in space. Then, by a well know procedure due to Burq
  \cite{Burq03-a}, it is possible to deduce from the
  smoothing estimate \eqref{eq:smooschr} the full set of
  \emph{global in time} (non endpoint) Strichartz estimates
  for the same equation.

  The same remark applies to the wave equation
  \eqref{eq:wave} which we consider next.
\end{remark}

\begin{corollary}[]\label{cor:smoowe}
  Assume  $L=-A^{b}+c$ and $\Omega$ are
  as in the previous corollary. Then we have the estimates
  \begin{equation}\label{eq:smoowe1}
    \|\bra{x}^{-\frac12-}e^{it \sqrt{L}}f\|
        _{L^{2}_{t}L^{2}_{x}(\Omega)}
    \lesssim
    \|f\|_{L^{2}(\Omega)},
  \end{equation}
  \begin{equation}\label{eq:smoowe2}
    \|\bra{x}^{-1-}e^{it \sqrt{L}}f\|
        _{L^{2}_{t}L^{2}_{x}(\Omega)}
    \lesssim
    \|f\|_{\dot H^{\frac12}(\Omega)},
  \end{equation}
  \begin{equation}\label{eq:smoowenh1}
    \textstyle
    \|\int_{0}^{t}\bra{x}^{-\frac12-}
           e^{i(t-s)\sqrt{L}}Fds\|
          _{L^{2}_{t}L^{2}_{x}(\Omega)}
    \lesssim
    \|\bra{x}^{\frac12+} F\|
          _{L^{2}_{t}L^{2}_{x}(\Omega)}.
  \end{equation}
  \begin{equation}\label{eq:smoowenh2}
      \textstyle
      \|\int_{0}^{t}\bra{x}^{-1-}e^{i(t-s)\sqrt{L}}Fds\|
            _{L^{2}_{t}L^{2}_{x}(\Omega)}
      \lesssim
      \|\bra{x}^{1+}L^{\frac12}F\|
            _{L^{2}_{t}L^{2}_{x}(\Omega)},
  \end{equation}
  Moreover, for every $f\in L^{2}(\Omega)$,
  we have the RAGE type property
  \begin{equation}\label{eq:ragewe}
    \|\bra{x}^{-\frac12-}e^{it \sqrt{L}}f\|_{L^{2}(\Omega)}\to0
    \quad\text{as}\quad 
    t\to\pm \infty.
  \end{equation}
\end{corollary}

\begin{remark}[]\label{rem:directmorawetxWE}
  As for the Schr\"{o}dinger equation, a direct application
  of the multiplier method to the wave equaiton
  \eqref{eq:wave} gives an estimate of the form
  \begin{equation*}
    \|\nabla e^{itL}f\|_{\dot Y L^{2}_{t}}
      +\|e^{{it \sqrt{L}}}f\|_{\dot X L^{2}_{t}}
    \lesssim
    \|f\|_{\dot H^{1}}.
  \end{equation*}
\end{remark}

\begin{remark}[]\label{rem:bibliosmoo}
  There is a large literature on smoothing properties
  for Schr\"{o}dinger, wave and Dirac equations with 
  electric and magnetic potentials;
  some results are available also
  in the case of fully variable coefficients,
  both local and global in time
  (see e.g.
  \cite{Yajima90-a},
  \cite{Pravica93-a},
  \cite{CraigKappelerStrauss95-a},
  \cite{KapitanskiSafarov96-a},
  \cite{RobbianoZuily98-b},
  \cite{Doi00-a},
  \cite{SmithSogge00-a},
  \cite{Alinhac06-a},
  \cite{MarzuolaMetcalfeTataru08-a},
  \cite{RodnianskiTao11-a},
  \cite{Tataru13-a},
  \cite{VasyWunsch13-a},
  \cite{DAnconaFanelli07-a},
  \cite{CacciafestaDAncona13-a}
  and the references therein).
  On the other hand, the case of exterior domains is less
  studied: we mention at least the papers
  \cite{Tsutsumi93-a},
  \cite{BurqGerardTzvetkov04-c},
  \cite{MetcalfeSogge06-a},
  \cite{RobbianoZuily09-a},
  \cite{MetcalfeTataru09-a}.
\end{remark}

\section{Basic identities}\label{sec:identities}

In this section we show how to adapt the method 
of Morawetz multipliers
to a general elliptic operator with variable coefficients.
Using implicit summation over repeated indices,
we can write the principal part in the form
\begin{equation*}
  A^{b}v:=\nabla^b\cdot(a(x)\nabla^bv)=
  \partial_j^b(a_{jk}(x)\partial_k^bv).
\end{equation*}
We also write
\begin{equation*}
  Av:=\nabla\cdot(a(x)\nabla v)=
  \partial_j(a_{jk}(x)\partial_kv)
\end{equation*}
and we recall the notations
\begin{equation*}
  \widehat{x}=\frac{x}{|x|},\qquad
  a(w,z)=a_{jk}(x)w_k\overline{z}_j.
\end{equation*}
The appropriate multiplier for the operator $A^b$ is the quantity
\begin{equation}\label{eq:mormult}
  [A^{b},\psi]\overline{v}=(A \psi)\overline{v}
  +2a(\nabla \psi,\nabla^{b}v).
\end{equation}
The following identities are based on the multiplier
\eqref{eq:mormult} and the simpler multiplier 
$\phi \overline{v}$:

\begin{proposition}\label{pro:mor}
  Assume $a=[a_{jk}]$ is symmetric and real valued and
  $b$ is real valied,
  while $V(x)$ and $v(x)$ may be complex valued,
  and all functions are sufficiently smooth. 
  Then, denoting with $Q_{j}$ and $P_{j}$ the quantities
  \begin{equation*}
    Q_{j}=
    a_{jk}\partial^{b}_{k}v \cdot 
      [A^{b},\psi]\overline{v}
      -\textstyle\frac12 a_{jk}(\partial_{k}A \psi)|v|^{2}
      -a_{jk}\partial_{k}\psi 
      \left[V|v|^{2}+a(\nabla^{b}v,\nabla^{b}v)\right]
  \end{equation*}
  and
  \begin{equation}
    P_{j}=
    a_{jk}\partial^{b}_{k}v\phi \overline{v}
       -\textstyle \frac12 a_{jk}\partial_{k}\phi|v|^{2}
  \end{equation}
  for arbitrary real weights $\psi(x),\phi(x)$,
  the following identities hold:
  \begin{equation}\label{eq:id1}
  \begin{split}
    \Re \partial_{j}Q_{j}=
    &
    \Re (A^{b}v-Vv)
    [A^{b},\psi]\overline{v}
    \\
    &
    -\textstyle \frac12 A^{2}\psi|v|^{2}
    + \Re(\alpha_{\ell m}\ 
      \partial^{b}_{m}v\ \overline{\partial^{b}_{\ell}v})
    -  \Re a(\nabla \psi,\nabla V)|v|^{2}
    \\
    &
    +2\Im(a_{jk}\partial^{b}_{k}v
       (\partial_{j}b_{\ell}-\partial_{\ell}b_{j})
       a_{\ell m}\partial_{m}\psi\ \overline{v})
    -2(\Im V)\Im(a(\nabla \psi,\nabla^{b}v)v)
  \end{split}
  \end{equation}
  where
  \begin{equation*}%
    \alpha_{\ell m}=2a_{jm}\partial_{j}(a_{\ell k}
       \partial_{k}\psi)
            -a_{jk}\partial_{k}\psi \partial_{j}a_{\ell m},
  \end{equation*}
  and
  \begin{equation}\label{eq:id2}
    \partial_{j}P_{j}
    =
    A^{b}v \cdot \phi \overline{v}
    +a(\nabla^{b}v,\nabla^{b}v)\phi
    -\textstyle \frac12 A \phi|v|^{2}
    +i\Im a(\nabla^{b} v,v \nabla \phi).
  \end{equation}
\end{proposition}

\begin{proof}
  As mentioned above, both identities can be deduced by 
  multiplying
  the quantity $A^{b}v-V(x)v$ by \eqref{eq:mormult} or 
  $\phi \overline{v}$ respectively. The computations are
  long but elementary, and 
  once the identities \eqref{eq:id1} and 
  \eqref{eq:id2} are known, it is straightforward
  to check their validity. We omit the details.
\end{proof}

Applying the two identities of the previous Proposition
to a solution of the Helmholtz equation
with $V=c(x)-\lambda-i \epsilon$, then adding them,
after a few elementary manipulations we obtain:

\begin{theorem}\label{the:morid}
  Assume $v\in H^{2}_{loc}(\Omega)$ satisfies
  \begin{equation}
    A^{b}v-c(x)v+(\lambda+i \epsilon)v=f(x)
  \end{equation}
  for $a,b,c$ real valued, $a(x)$ symmetric,
  $\lambda,\epsilon\in \mathbb{R}$. 
  Then the following identities hold for any
  real weights $\phi(x),\psi(x)$:
  \begin{equation}\label{eq:morid}
  \begin{split}
    \Re \partial_{j}\{Q_{j}+P_{j}\}=
    &-\textstyle \frac12 A(A \psi+\phi)|v|^{2}
    -[a(\nabla \psi,\nabla c)-c \phi+\lambda \phi]|v|^{2}
    \\
    &+ \Re[\alpha_{\ell m}\ 
      \partial^{b}_{m}v\ \overline{\partial^{b}_{\ell}v}]
    +a(\nabla^{b}v,\nabla^{b}v)\phi
    \\
    &+2 \epsilon\Im[a(\nabla \psi,\nabla^{b}v)v]
    +2\Im[a_{jk}\partial^{b}_{k}v
       (\partial_{j}b_{\ell}-\partial_{\ell}b_{j})
       a_{\ell m}\partial_{m}\psi\ \overline{v}]
    \\
    &+\Re [(A\psi+\phi)\overline{v}f+2a(\nabla\psi,\nabla^bv)f]
  \end{split}
  \end{equation}
  and
  \begin{equation}\label{eq:morid2}
    \partial_{j}P_{j}=
    a(\nabla^{b}v,\nabla^{b}v)\phi
    +(c-\lambda-i \epsilon)|v|^{2}\phi
    +f \overline{v}\phi
    -\textstyle \frac12 A \phi|v|^{2}
    +i\Im a(\nabla^{b} v,v \nabla \phi),
  \end{equation}
  where 
  \begin{equation*}
    P_{j}=
    a_{jk}\partial^{b}_{k}v\phi \overline{v}
       -\textstyle \frac12 a_{jk}\partial_{k}\phi|v|^{2}
  \end{equation*}
  \begin{equation*}
    Q_{j}=
    a_{jk}\partial^{b}_{k}v \cdot 
      [A^{b},\psi]\overline{v}
      -\textstyle\frac12 a_{jk}(\partial_{k}A \psi)|v|^{2}
      -a_{jk}\partial_{k}\psi 
      \left[(c-\lambda)|v|^{2}+a(\nabla^{b}v,\nabla^{b}v)\right]
  \end{equation*}
  and
  \begin{equation*}
    \alpha_{\ell m}=2a_{jm}\partial_{j}(a_{\ell k}
         \partial_{k}\psi)
          -a_{jk}\partial_{k}\psi \partial_{j}a_{\ell m}.
  \end{equation*}
\end{theorem}

\section{Morrey-Campanato type norms and their properties}
\label{sec:norms}

In this section we prove some relations
between the Morrey-Campanato type norms $\dot X,\dot Y,X,Y$
and usual weighted $L^{2}$ norms.
If $\Omega$ is an open subset of $\mathbb{R}^{n}$, $n\ge2$,
we write
\begin{equation*}
    \Omega_{=R}=\Omega\cap\{x \colon |x|=R\},\quad
    \Omega_{\le R}=\Omega\cap\{x \colon |x|\le R\},\quad
    \Omega_{\ge R}=\Omega\cap\{x \colon |x|\ge R\},
\end{equation*}
\begin{equation*}
    \Omega_{R_{1}\le|x|\le R_{2}}
    =\Omega\cap\{x \colon R_{1}\le|x|\le R_{2}\}.
\end{equation*}
The homogeneous norm $\dot X$ and the corresponding
predual norm $\dot X^{*}$ of a function $v:\Omega\to \mathbb{C}$
are defined as
\begin{equation*}
  \textstyle
  \|v\|_{\dot X}^{2}:=
  \sup\limits_{R>0}\frac{1}{R^{2}}\int_{\Omega_{=R}}
  |v|^{2}dS,
  \qquad
  \|v\|_{\dot X^{*}}:=
  \int_{0}^{+\infty}
  r
  \left(\int_{\Omega_{=r}}|v|^{2}dS\right)^{\frac12}dr.
\end{equation*}
where $dS$ is the surface measure on $\Omega_{=R}$.
The corresponding nonhomogeneous versions, denoted
by $X,X^{*}$, are
\begin{equation}\label{eq:MC2}
  \textstyle
  \|v\|_{X}^{2}:=
  \sup\limits_{R>0}\frac{1}{\langle R\rangle^{2}}
  \int_{\Omega_{=R}}
  |v|^{2}dS,
  \qquad
  \|v\|_{X^{*}}:=
  \int_{0}^{+\infty}
  \bra{r}
  \left(\int_{\Omega_{=r}}|v|^{2}dS\right)^{\frac12}dr
\end{equation}
where $\langle R\rangle=\sqrt{1+R^{2}}$.
We shall also need proper Morrey-Campanato spaces, both in
the homogeneous version $\dot Y$ and in the non homogenous
version $Y$; their norms are defined as
\begin{equation}\label{eq:MC3}
  \textstyle
  \|v\|_{\dot Y}^{2}:=
  \sup\limits_{R>0}\frac{1}{R}\int_{\Omega_{\le R}}
  |v|^{2}dx,
  \qquad
  \|v\|_{Y}^{2}:=
  \sup\limits_{R>0}
  \frac{1}{\langle R\rangle}\int_{\Omega_{\le R}}
  |v|^{2}dx.
\end{equation}
The following equivalence is easy to prove:
\begin{equation}\label{eq:equivnonhom}
  \textstyle
  \|v\|_{Y}^{2}\le
  \sup_{R\ge1}\frac1R\int_{\Omega_{\le R}}|v|^{2}\le
  \sqrt{2}\|v\|_{Y}^{2}.
\end{equation}
A concrete characterization of the predual norms $\dot Y^{*}$,
$Y^{*}$ is less immediate. The simplest approach is
to introduce an equivalent dyadic norm
\begin{equation*}
\textstyle
  \|v\|_{ \dot Y_{\Delta} }:=
  \sup _{j\in \mathbb{Z}}
  2^{-j/2}\|v\|_{L^{2}(\Omega_{2^{j-1}\le|x|\le 2^{j}})}
\end{equation*}
which satisfies, as it is readily seen,
\begin{equation*}
  2^{-1}\|v\|_{ \dot Y_{\Delta} }
  \le
  \|v\|_{ \dot Y }
  \le
  2\|v\|_{ \dot Y_{\Delta} }
\end{equation*}
and similarly
\begin{equation*}
  \textstyle
  \|v\|_{ Y_{\Delta} }:=
  \|v\|_{L^{2}(\Omega_{\le1})}+
  \sup _{j\ge1}
  2^{-j/2}\|v\|_{L^{2}(\Omega_{2^{j-1}\le|x|\le 2^{j}})},
\end{equation*}
which satisfies
\begin{equation*}
  3^{-1}\|v\|_{ Y_{\Delta} }
  \le
  \|v\|_{ Y }
  \le
  3\|v\|_{ Y_{\Delta} }.
\end{equation*}
We then obtain the following characterizations:
\begin{equation*}
  \|v\|_{\dot Y^{*}}
  \simeq
  \sum_{j\in \mathbb{Z}}2^{j/2}
     \|v\|_{L^{2}(\Omega _{2^{j-1}\le|x|\le 2^{j}})},
  \quad
  \|v\|_{Y^{*}}
  \simeq
  \|v\|_{L^{2}(\Omega_{\le 1})}
  +
  \sum_{j\ge1}2^{j/2}
  \|v\|_{L^{2}(\Omega _{2^{j-1}\le|x|\le 2^{j}})}.
\end{equation*}
The following Lemmas contain several estimates to be used
in the rest of the paper.

\begin{lemma}\label{lem:normest1}
  For any $v\in C^{\infty}_{c}(\mathbb{R}^{n})$,
  \begin{equation}\label{eq:stnorma8}
    \||x|^{-1}v\|_{\dot Y}\le
    \|v\|_{\dot X},\qquad
    \|\bra{x}^{-1}v\|_{ Y}\le
    \|v\|_{ X}.
  \end{equation}
\end{lemma}

\begin{proof}
  Both inequalities are immediate in polar coordinates:
  for the first we have
  \begin{equation*}
    \textstyle
    \sup_{R>0}\frac1R\int_{\Omega_{\le R}}\frac{|v|^{2}}{|x|^{2}}
    =
    \sup_{R>0}\frac1R
    \int_{0}^{R}\frac{1}{\rho^{2}}
    \int_{\Omega_{=\rho}}|v|^{2}dSd\rho
    \le
    \|v\|_{\dot X}^{2}
    \sup_{R>0}\frac1R\int_{0}^{R}d\rho
\end{equation*}
  and the second one is similar.
\end{proof}

\begin{lemma}\label{lem:normest2}
  For any $0<\delta<1$ and $v\in C^{\infty}_{c}(\mathbb{R}^{n})$,
  \begin{equation}\label{eq:stnorma1}
    \textstyle
    \int_{\Omega}
    \frac{|v|^{2}}{|x|^{2}\langle x\rangle^{1+\delta}}
    \le  2  \delta^{-1}\|v\|_{\dot X}^{2},
  \end{equation}
  \begin{equation}\label{eq:stnorma3}
    \textstyle
    \int_{\Omega_{\ge1}}
    \frac{|v|^{2}}{|x|^{3}\langle x\rangle^{\delta}}
    \le
    \int_{\Omega_{\ge1}}
    \frac{|v|^{2}}{|x|^{3+\delta}}
    \le  
    2 \delta^{-1}\|v\|_{X}^{2},
  \end{equation}
  \begin{equation}\label{eq:stnorma4}
    \textstyle
    \int_{\Omega}\frac{|v|^{2}}{\langle x\rangle^{1+\delta}}
    \le
    8 \delta^{-1}
    \|v\|_{Y}^{2}
    \le
    8 \delta^{-1}
    \|v\|_{\dot Y}^{2}.
  \end{equation}
\end{lemma}

\begin{proof}
  To prove \eqref{eq:stnorma1}, \eqref{eq:stnorma3} we write

  \begin{equation*}
    \textstyle
    \int_{\Omega}
    \frac{|v|^{2}}{|x|^{2}\langle x\rangle^{1+\delta}}
    \le
    \|v\|_{\dot X}^{2}
    \int_{0}^{\infty}\bra{\rho}^{-1-\delta}d\rho,
    \qquad
    \int_{\Omega_{\ge1}}
    \frac{|v|^{2}}{|x|^{3+\delta}}
    \le
    \|v\|_{X}^{2}
    \int_{1}^{\infty}
    \frac{\bra{\rho}^{2}}{\rho^{3+\delta}}d \rho
  \end{equation*}
  and notice that $\bra{\rho}\ge(1+\rho)/\sqrt{2}$,
  and that $\bra{\rho}^{2}\le 2 \rho^{2}$ for $\rho\ge1$.
  To prove \eqref{eq:stnorma4} we split
  $\Omega=\Omega_{\le1}\cup \Omega_{\ge1}$;
  we have immediately for the first piece
  \begin{equation*}
    \textstyle
    \int_{\Omega_{\le1}}
    \frac{|v|^{2}}{\langle x\rangle^{1+\delta}}
    \le
    \int_{\Omega_{\le1}}|v|^{2}
    \le
    \sqrt{2}\|v\|_{Y}^{2}.
  \end{equation*}
  For the remaining piece we use a dyadic decomposition
  $\Omega_{\ge1}=\cup_{j\ge0}\Omega_{2^{j}\le|x|\le2^{j+1}}$,
  we notice that
  \begin{equation*}
    \textstyle
    \int_{\Omega_{2^{j}\le|x|\le 2^{j+1}}}
    \frac{|v|^{2}}{\langle x\rangle^{1+\delta}}
    \le
    \frac{\bra{2^{j+1}}}{\bra{2^{j}}^{1+\delta}}
    \frac{1}{\bra{2^{j+1}}}
    \int_{\Omega_{\le2^{j+1}}}|v|^{2}
    \le
    \|v\|_{Y}^{2}\frac{\bra{2^{j+1}}}{\bra{2^{j}}^{1+\delta}}
  \end{equation*}
  and we sum over $j\ge0$. Since
  \begin{equation*}
    \textstyle
    \sum_{j\ge0}
    \frac{\bra{2^{j+1}}}{\bra{2^{j}}^{1+\delta}}
    \le
    2^{3/2}\sum_{j\ge0}2^{-j \delta}
    \le
    \frac{2^{3/2+\delta}}{2^{\delta}-1}
    \le
    2^{5/2}\delta^{-1}
  \end{equation*}
  we obtain \eqref{eq:stnorma4} with a constant
  $\sqrt{2}+ 2^{5/2}\le8$.
\end{proof}

\begin{lemma}\label{lem:normest3}
  For any $v\in C^{\infty}_{c}(\mathbb{R}^{n})$,
  \begin{equation}\label{eq:stnorma6}
    \textstyle
    \sup_{R>1}\frac{1}{R^{2}}
    \int_{\Omega_{R\le|x|\le 2R}}
    \frac{|v|^{2}}{|x|}dx\le
    3\|v\|_{X}^{2},
  \end{equation}
  \begin{equation}\label{eq:stnorma9}
    \textstyle
    \sup_{R>0}\frac{1}{R^{2}}\int_{\Omega_{R\le|x|\le2R}}
    \frac{|v|^{2}}{|x|}
    \le
    \frac{3}{2}\|v\|_{\dot X}^{2},
  \end{equation}
  \begin{equation}\label{eq:stnorma2}
    \textstyle
    \sup\limits_{R>0}
    \int_{\Omega_{\ge R}}
    \frac{R^{n-1}}{|x|^{n+2}}|v|^{2}dx
    \le
    \frac{1}{n-1} \|v\|_{\dot X}^{2},
    \qquad
    \sup\limits_{R>1}
    \int_{\Omega_{\ge R}}
    \frac{R^{n-1}}{|x|^{n+2}}|v|^{2}dx
    \le
    \frac{2}{n-1} \|v\|_{X}^{2}
  \end{equation}
\end{lemma}

\begin{proof}
  Trivial.
\end{proof}

\begin{lemma}\label{lem:normest4}
  For any $R>0$, $0<\delta<1$ and 
  $v,w\in C^{\infty}_{c}(\mathbb{R}^{n})$,
  \begin{equation}\label{eq:stnorma5}
    \textstyle
    \int_{\Omega_{R\le|x|\le2R}}
    |vw|dx\le
    3R^{2}\|v\|_{\dot X}\|w\|_{\dot Y},\qquad
    \int_{\Omega_{\le R}}
    |vw|dx\le
    R\|v\|_{\dot X}\|w\|_{\dot Y^{*}},
  \end{equation}
  \begin{equation}\label{eq:stnorma5b}
    \textstyle
    \int_{\Omega_{R\le|x|\le2R}}
    |vw|dx\le
    3 \langle R\rangle ^{2}\|v\|_{ X}\|w\|_{ Y},\qquad
    \int_{\Omega_{\le R}}
    |vw|dx\le
    \langle R\rangle\|v\|_{ X}\|w\|_{ Y^{*}},
  \end{equation}
  \begin{equation}\label{eq:stnorma10}
    \textstyle
    \int_{\Omega_{\le1}}\frac{|vw|}{|x|^{2-\delta}}
    +
    \int_{\Omega_{\ge1}}\frac{|vw|}{|x|^{2+\delta}}
    \le
    9\delta^{-1}
    \|v\|_{\dot X}\|w\|_{\dot Y},
  \end{equation}
  \begin{equation}\label{eq:stnorma10b}
    \textstyle
    \int_{\Omega_{\ge1}}\frac{|vw|}{|x|^{2+\delta}}
    \le
    12 \delta^{-1}
    \|v\|_{ X}\|w\|_{ Y}.
  \end{equation}
\end{lemma}

\begin{proof}
  The first inequalities in \eqref{eq:stnorma5} and
  \eqref{eq:stnorma5b} are trivial;
  the second ones follow by duality and by the estimates
  $\|\one{\Omega_{\le R}}v\|_{\dot Y}\le R\|v\|_{\dot X}$ and
  $\|\one{\Omega_{\le R}}v\|_{ Y}\le \bra{R}\|v\|_{ X}$,
  where $\one{K}$ denotes the characteristic function of a
  set $K$. To prove \eqref{eq:stnorma10}, 
  we use \eqref{eq:stnorma5} to write
  \begin{equation*}
    \textstyle
    \int_{\Omega_{\le1}}\frac{|vw|}{|x|^{2-\delta}}
    \le
    \sum_{j\ge1}
    2^{j(2-\delta)}
    \int_{\Omega_{2^{-j}\le|x|\le 2^{-j+1}}}
    |vw|
    \le
    \|v\|_{\dot X}\|w\|_{\dot Y}
    \sum_{j\ge1}
    2^{j(2-\delta)} \cdot 3 \cdot 2^{-2j}
  \end{equation*}
  which gives
  \begin{equation*}
    \textstyle
    \int_{\Omega_{\le1}}\frac{|vw|}{|x|^{2-\delta}}
    \le
    \frac{3}{2^{\delta}-1}
    \|v\|_{\dot X}\|w\|_{\dot Y}
  \end{equation*}
  and similarly
  \begin{equation*}
    \textstyle
    \int_{\Omega_{\ge1}}\frac{|vw|}{|x|^{2+\delta}}
    \le
    \sum_{j\ge0}
    2^{-j(2+\delta)}
    \int_{\Omega_{2^{j}\le|x|\le 2^{j+1}}}
    |vw|
    \le
    \frac{3 \cdot 2^{\delta}}{2^{\delta}-1}
    \|v\|_{\dot X}\|w\|_{\dot Y}
  \end{equation*}
  so that, summing up,
  \begin{equation*}
    \textstyle
    \int_{\Omega_{\le1}}\frac{|vw|}{|x|^{2-\delta}}
    +
    \int_{\Omega_{\ge1}}\frac{|vw|}{|x|^{2+\delta}}
    \le
    3 \frac{2^{\delta}+1}{2^{\delta}-1}
    \|v\|_{\dot X}\|w\|_{\dot Y}
    \le
    \frac 9 \delta 
    \|v\|_{\dot X}\|w\|_{\dot Y}.
  \end{equation*}
  For the nonhomogeneous estimate \eqref{eq:stnorma10b}
  we write, using \eqref{eq:stnorma5b},
  \begin{equation*}
    \textstyle
    \int_{\Omega_{\ge1}}\frac{|vw|}{|x|^{2+\delta}}
    \le
    \sum_{j\ge0}
    2^{-j(2+\delta)}
    \int_{\Omega_{2^{j}\le|x|\le 2^{j+1}}}
    |vw|
    \le
    3 \sum_{j\ge0}
    2^{-j(2+\delta)}\bra{2^{j}}^{2}
    \cdot
    \|v\|_{X}\|w\|_{Y}
  \end{equation*}
  and notice that
  \begin{equation*}
    \textstyle
    \sum_{j\ge0} 2^{-j(2+\delta)}\bra{2^{j}}^{2}
    \le
    2\sum_{j\ge0}2^{-j \delta}
    =
    \frac{2^{\delta+1}}{2^{\delta}-1}\le \frac4 \delta.
  \end{equation*}
\end{proof}

In the following Lemma we prove some magnetic
Hardy type inequalities, which require $n\ge3$,
expressed in terms of the nonhomogeneous $X,Y$ norms:

\begin{lemma}\label{lem:normest5}
  Let $n\ge3$ and
  assume $b(x)=(b_{1}(x),\dots,b_{n}(x))$ is continuous
  up to the boundary of $\Omega$ with values in $\mathbb{R}^{n}$.
  For any $0<\delta<1$ and
  $v\in C^{\infty}_{c}(\Omega)$, we have:
  \begin{equation}\label{eq:hardymag1}
    \textstyle
    \||x|^{-1}v\|_{L^{2}(\Omega)}\le
    \frac{2}{n-2}\|\nabla^{b}v\|_{L^{2}(\Omega)},
  \end{equation}
  \begin{equation}\label{eq:stnorma12}
    \||x|^{-1}v\|^{2}_{Y}\le
    6\|\nabla^{b}v\|^{2}_{Y}+ 3\|v\|^{2}_{X},
  \end{equation}
  \begin{equation}\label{eq:stnorma13}
    \textstyle
    \int_{\Omega_{\le1}} \frac{|\nabla^{b}v| |v|}{|x|}dx+
    \int_{\Omega_{\ge1}}\frac{|\nabla^{b}v| |v|}{|x|^{2+\delta}}dx
    \le
    9 \delta^{-1}
    (\|\nabla^{b}v\|^{2}_{Y}+\|v\|^{2}_{X}),
  \end{equation}
  \begin{equation}\label{eq:stnorma7}
    \textstyle
    \|v\|_{X}\le
    4\sup_{R>1}
    R^{-2} \int_{\Omega_{=R}}|v|^{2}dS
    +
    13\|\nabla^{b} v\|^{2}_{Y},
  \end{equation}
  \begin{equation}\label{eq:stnorma11}
    \textstyle
    \int_{\Omega}
    \frac{|v|^{2}}{|x|\bra{x}^{1+\delta}(1\vee |x|)}dx
    \le
    8\delta^{-1}\|v\|_{X}^{2}+
    9\int_{\Omega_{\le1}}|\nabla^{b}v|^{2}dx.
  \end{equation}
\end{lemma}

\begin{proof}
  We start from the identity
  \begin{equation*}
    \textstyle
    \nabla \cdot\left(\frac{x}{|x|^{2}}|v|^{2}\right)
    =
    \frac{n-2}{|x|^{2}}|v|^{2}+
    \frac{2}{|x|}\Re(\overline{v}\ \widehat{x}\cdot \nabla v)
    \equiv
    \frac{n-2}{|x|^{2}}|v|^{2}+
    \frac{2}{|x|}\Re(\overline{v}\ \widehat{x}\cdot \nabla^{b}v).
  \end{equation*}
  Integrating over $\Omega_{\le R}$ and noticing that
  $u\vert_{\partial \Omega}=0$, we get
  \begin{equation}\label{eq:stepA}
    \textstyle
    (n-2)\int_{\Omega_{\le R}}\frac{|v|^{2}}{|x|^{2}}dx
    =
    \frac1R\int_{\Omega_{=R}}|v|^{2}dS
    -\int_{\Omega_{\le R}}
    \frac{2}{|x|}\Re(\overline{v}\ \widehat{x}\cdot \nabla^{b}v)dx.
  \end{equation}
  Estimating the last term with Cauchy-Schwartz
  \begin{equation*}
    \textstyle
    -\int_{\Omega_{\le R}}
    \frac{2}{|x|}\Re(\overline{v}\ \widehat{x}\cdot \nabla^{b}v)d
    \le
    \frac{n-2}{2}
    \int_{\Omega_{\le R}} \frac{|v|^{2}}{|x|^{2}}dx
    +
    \frac{2}{n-2}
    \int_{\Omega_{\le R}}|\nabla^{b}v|^{2}dx
  \end{equation*}
  we obtain
  \begin{equation}\label{eq:stepB}
    \textstyle
    \int_{\Omega_{\le R}}\frac{|v|^{2}}{|x|^{2}}dx
    \le
    \frac{2}{n-2}
    \frac1R\int_{\Omega_{=R}}|v|^{2}dS
    +
    \frac{4}{(n-2)^{2}}
    \int_{\Omega_{\le R}}|\nabla^{b} v|^{2}dx.
  \end{equation}
  Letting $R\to \infty$ in \eqref{eq:stepB} we obtain
  \eqref{eq:hardymag1}. 
  On the other hand, for $0<R<1$ this gives
  \begin{equation}\label{eq:stepC}
    \textstyle
    \frac{1}{\bra{R}}
    \int_{\Omega_{\le R}}\frac{|v|^{2}}{|x|^{2}}dx
    \le
    2\int_{\Omega_{=1}}|v|^{2}dS
    +
    4 \int_{\Omega_{\le 1}}|\nabla^{b} v|^{2}dx
    \le
    2 \sqrt{2}\|v\|_{X}^{2}
    +
    4 \sqrt{2}\|\nabla^{b}v\|^{2}_{Y}
  \end{equation}
  while for $R>1$ it gives
  \begin{equation*}
    \textstyle
    \frac{1}{\bra{R}}
    \int_{\Omega_{\le R}}\frac{|v|^{2}}{|x|^{2}}dx
    \le
    \frac{2}{R \bra{R}}
    \int_{\Omega_{=R}}|v|^{2}dS
    +
    \frac{4}{\bra{R}}
    \int_{\Omega_{\le R}}|\nabla^{b} v|^{2}dx
    \le
    2 \sqrt{2}\|v\|_{X}^{2}
    +
    4 \|\nabla^{b}v\|^{2}_{Y}.
  \end{equation*}
  The last two inequalities together imply \eqref{eq:stnorma12}.
  To prove \eqref{eq:stnorma13}, we split 
  $\Omega=\Omega_{\le1}\cup \Omega_{\ge1}$ and we
  remark that
  \begin{equation*}
    \textstyle
    \int_{\Omega_{\le1}}
    \frac{|\nabla^{b}v| |v|}{|x|}dx
    \le
    \frac{1}{\sqrt{2}}
    \int_{\Omega_{\le1}}|\nabla^{b}v|^{2}+
    \frac{1}{2 \sqrt{2}}
    \int_{\Omega_{\le1}}\frac{|v|^{2}}{|x|^{2}}
    \le
    3\|\nabla^{b}v\|^{2}_{Y}
    +
    \sqrt{2}\|v\|_{X}^{2}
  \end{equation*}
  using \eqref{eq:stepC}, while using \eqref{eq:stnorma10b}
  \begin{equation*}
    \textstyle
    \int_{\Omega_{\ge1}}
    \frac{|\nabla^{b}v| |v|}{|x|^{2+\delta}}dx
    \le
    6 \delta^{-1}\|\nabla^{b}v\|^{2}_{Y}
    +
    6 \delta^{-1 }\|v\|_{X}^{2}.
  \end{equation*}
  Summing up, we obtain \eqref{eq:stnorma13}.

  On the other hand, from \eqref{eq:stepA} we can deduce also,
  for all $0<R<1$,
  \begin{equation*}
  \begin{split}
    \textstyle
    \frac1R\int_{\Omega_{=R}}|v|^{2}dS
    \le&
    \textstyle
    (n-2)\int_{\Omega_{\le R}}\frac{|v|^{2}}{|x|^{2}}dx
    +
    \int_{\Omega_{\le R}}\frac{|v|^{2}}{|x|^{2}}dx
    +
    \int_{\Omega_{\le R}}|\nabla^{b}v|^{2}dx
    \\
    \le&
    \textstyle
    (n-1)\int_{\Omega_{\le 1}}\frac{|v|^{2}}{|x|^{2}}dx
    +
    \int_{\Omega_{\le 1}}|\nabla^{b}v|^{2}dx.
  \end{split}
  \end{equation*}
  Together with \eqref{eq:stepB} with $R=1$, this gives,
  for all $0<R<1$,
  \begin{equation*}
    \textstyle
    \frac1R\int_{\Omega_{=R}}|v|^{2}dS
    \le
    \frac{2(n-1)}{(n-2)}
    \int_{\Omega_{=1}}|v|^{2}dS
    +\left(
    \frac{4(n-1)}{(n-2)^{2}}+1
    \right)
    \int_{\Omega_{\le1}}|\nabla^{b}v|^{2}dx
  \end{equation*}
  and recalling that $n\ge3$ we have proved for all $0<R<1$
  \begin{equation*}
    \textstyle
    \frac1R\int_{\Omega_{=R}}|v|^{2}dS\le
    4\int_{\Omega_{=1}}|v|^{2}dS +
    9\int_{\Omega_{\le1}}|\nabla^{b} v|^{2}dx.
  \end{equation*}
  from which \eqref{eq:stnorma7} and
  \eqref{eq:stnorma11} follow easily.
\end{proof}

By a density argument, it is clear that the estimates in
Lemmas \ref{lem:normest1}--\ref{lem:normest5} are valid
not only for smooth functions but also for functions belonging
to the domain of the operator $D(-A^{b}+c)$; in particular,
Lemmas \ref{lem:normest5} hold
in view of the Dirichlet boundary conditions in the definition
of the operator.

\section{Proof of the Theorems}\label{sec:proofsthms}

The proof consists in integrating the identity \eqref{eq:morid}
on $\Omega$ and estimating all the terms. Since
the arguments for both Theorems \ref{the:1}
and \ref{the:2} largely overlap, we shall proceed with
both proofs in parallel. The proof is divided into
several steps.

\subsection{Notations}\label{sub:notations}

Recall that we are using 
\emph{implicit summation over repeated indices} and
the notations
\begin{equation*}
  \textstyle
  \widehat{x}=\frac{x}{|x|}=
  (\widehat{x}_{1},\dots,\widehat{x}_{n}),\qquad
  \widehat{x}_{j}=\frac{x_{j}}{|x|},
\end{equation*}
\begin{equation*}
  \widehat{a}(x)=a_{\ell m}(x)\widehat{x}_{\ell}\widehat{x}_{m},
  \qquad
  \overline{a}(x)=\mathop{\textrm{trace}}a(x)=a_{mm}(x).
\end{equation*}
Since $a(x)$ is positive definite, we have
\begin{equation*}
  0\le \widehat{a}= a \widehat{x}\cdot \widehat{x}
  \le |a \widehat{x}|\le \overline{a}.
\end{equation*}
We denote with a semicolon the partial derivatives:
\begin{equation*}
  a_{jk;\ell}:=\partial_{\ell}a_{jk},\qquad
  a_{jk;\ell m}:=\partial_{\ell}\partial_{m}a_{jk},\qquad
  a_{jk;\ell mp}:=\partial_\ell\partial_ m\partial_p a_{jk}.
\end{equation*}
Notice the formulas
\begin{equation*}
  \partial_{k}(\widehat{x}_{\ell})=
  |x|^{-1}[\delta_{k\ell}-\widehat{x}_{k}\widehat{x}_{\ell}],
\end{equation*}
\begin{equation*}
  \partial_{k}(\widehat{x}_{\ell}\widehat{x}_{m})=
  |x|^{-1}[\delta_{k\ell}\widehat{x}_{m}+\delta_{km}
     \widehat{x}_{\ell}
  -2 \widehat{x}_{k} \widehat{x}_{\ell}\widehat{x}_{m}],
\end{equation*}
\begin{equation*}
\begin{split}
  \textstyle
  \partial_{j}\partial_{k}(\widehat{x}_{\ell}\widehat{x}_{m})=
  &
  \frac{1}{|x|^{2}}
  [
  \delta_{k\ell}\delta_{jm}+\delta_{km}\delta_{j\ell}
  +8 \widehat{x}_{j}\widehat{x}_{k}\widehat{x}_{\ell}
     \widehat{x}_{m}
    \\
  &
  -2 \delta_{k\ell}\widehat{x}_{j}\widehat{x}_{m}
  -2 \delta_{km}\widehat{x}_{j}\widehat{x}_{\ell}
  -2\delta_{jk}\widehat{x}_{\ell}\widehat{x}_{m}
  -2\delta_{j\ell}\widehat{x}_{k}\widehat{x}_{m}
  -2\delta_{jm}\widehat{x}_{k}\widehat{x}_{\ell}
  ]
\end{split}
\end{equation*}
which imply
\begin{equation*}
  a_{jk}a_{\ell m}\widehat{x}_{j}
  \partial_{k}(\widehat{x}_{\ell}\widehat{x}_{m})=
  2|x|^{-1}[|a \widehat{x}|^{2}-\widehat{a}^{2}],
\end{equation*}
and
\begin{equation*}
  \textstyle
  a_{jk}a_{\ell m}
  \partial_{j}\partial_{k}(\widehat{x}_{\ell}\widehat{x}_{m})=
  \frac{2}{|x|^{2}}
  [
  a_{\ell m}a_{\ell m}
  -4(|a \widehat{x}|^{2}-\widehat{a}^{2})
  -\overline{a}\widehat{a}
  ].
\end{equation*}
By the previous identities, for any
radial function $\psi(x)=\psi(|x|)$ we can write
\begin{equation}\label{eq:apsigen}
  A \psi(x)=\partial_{\ell}(a_{\ell m}\widehat{x}_{m}\psi')
  =
  \widehat{a}\psi''+
  \frac{\overline{a}-\widehat{a}}{|x|}
  \psi'
  +
  a_{\ell m;\ell}\widehat{x}_{m}\psi'
\end{equation}
where $\psi'$ denotes the derivative of $\psi(r)$ with respect
to the radial variable.

\subsection{Estimate of the \texorpdfstring{$\epsilon$}{} term}
\label{sub:epsilonterm}

We begin with the $\epsilon$-term in \eqref{eq:morid}
\begin{equation*}
  I_{\epsilon}:=2 \epsilon\Im[a(v\nabla \psi,\nabla^{b}v)].
\end{equation*}
We need a few auxiliary estimates.
Choosing $\phi=1$ in \eqref{eq:morid2} and taking the real part
we get
\begin{equation*}
  a(\nabla^{b}v,\nabla^{b}v)=
  (\lambda-c)|v|^{2}
  -\Re(f \overline{v})
  +\Re \partial_{j}\{\overline{v}a_{jk} \partial^{b}_{k}v\}
\end{equation*}
and assuming that
\begin{equation*}
  c(x)\ge - \sigma(x)^{2}
\end{equation*}
for some nonnegative function $\sigma(x)$ which will be precised
in the following (accordin to assumptions 
\eqref{eq:assc}--\eqref{eq:asscstrong}),
we obtain, with $\lambda_{+}=\max\{\lambda,0\}$,
\begin{equation}\label{eq:Rephi}
  a(\nabla^{b}v,\nabla^{b}v)\le
  (\lambda_{+}+\sigma^{2})|v|^{2}
  -\Re(f \overline{v})
  +\Re \partial_{j}\{\overline{v}a_{jk} \partial_{k}^{b}v\}.
\end{equation}
On the other hand taking the imaginary part of \eqref{eq:morid2} 
with $\phi=1$ gives
\begin{equation*}
  \epsilon|v|^{2}=
  \Im(f \overline{v})
  -\Im \partial_{j}\{\overline{v}a_{jk} \partial_{k}^{b}v\}
\end{equation*}
which can be written
\begin{equation}\label{eq:Imphi}
  |\epsilon||v|^{2}=
  \Im((\sgn \epsilon)f \overline{v})
  -\Im \partial_{j}\{(\sgn \epsilon)\overline{v}a_{jk} 
  \partial_{k}^{b}v\}.
\end{equation}
Now consider $I_{\epsilon}$
for an arbitrary function $\psi$ with bounded derivatives;
by Cauchy-Schwartz
\begin{equation*}
  |I_{\epsilon}|\le
  |\epsilon|a(v \nabla \psi,v \nabla \psi)
      (\lambda_{+}+\sigma^{2}+|\epsilon|)^{1/2}
  +
  |\epsilon|a(\nabla^{b}v,\nabla^{b}v)
      (\lambda_{+}+\sigma^{2}+|\epsilon|)^{-1/2}.
\end{equation*}
Since by $a(x)\le NI$ we have
\begin{equation*}%
  a(v \nabla \psi,v \nabla \psi)\le
  N\|\nabla\psi\|_{L^{\infty}}^{2}|v|^{2},
\end{equation*}
using \eqref{eq:Rephi} we obtain
\begin{equation*}%
\begin{split}
  \textstyle
  |I_{\epsilon}|
  \le
  N\|\nabla \psi\|^{2}_{L^{\infty}}
  {(\lambda_{+}+\sigma^{2}+|\epsilon|)^{1/2} }
  |\epsilon||v|^{2}
  +
  \frac{(\lambda_{+}+\sigma^{2})}
        {(\lambda_{+}+\sigma^{2}+|\epsilon|)^{1/2}}|
        \epsilon||v|^{2}
    \\
  +
  \textstyle
  \frac{|\epsilon|}{(\lambda_{+}+\sigma^{2}+|\epsilon|)^{1/2}}
  \Re(\partial_{j}\{\overline{v}a_{jk} \partial_{k}^{b}v\}-f 
  \overline{v})
\end{split}
\end{equation*}
and hence
\begin{equation*}
  \textstyle
  |I_{\epsilon}|
  \le
  (N\|\nabla \psi\|^{2}_{L^{\infty}}+1)
  {(\lambda_{+}+\sigma^{2}+|\epsilon|)^{1/2} }
  |\epsilon||v|^{2}
  +
  \frac{|\epsilon|}{(\lambda_{+}+\sigma^{2}+|\epsilon|)^{1/2}}
  \Re(\partial_{j}\{\overline{v}a_{jk} \partial_{k}^{b}v\}-f \overline{v}).
\end{equation*}
Using now \eqref{eq:Imphi} we get
\begin{equation*}
\begin{split}
  |I_{\epsilon}|\le
  (N\|\nabla \psi\|^{2}_{L^{\infty}}+1)
  {(\lambda_{+}+\sigma^{2}+|\epsilon|)^{1/2} }
  \Im[(\sgn \epsilon)f \overline{v})
  -\partial_{j}\{(\sgn \epsilon)\overline{v}a_{jk} 
  \partial_{k}^{b}v\}]
    \\
  \textstyle
  +
  \frac{|\epsilon|}{(\lambda_{+}+\sigma^{2}+|\epsilon|)^{1/2}}
  \Re(\partial_{j}\{\overline{v}a_{jk} 
  \partial_{k}^{b}v\}-f \overline{v})
\end{split}
\end{equation*}
which implies
\begin{equation*}
\begin{split}
  |I_{\epsilon}|\le
  &
  (N\|\nabla \psi\|^{2}_{L^{\infty}}+2)
  {(\lambda_{+}+\sigma^{2}+|\epsilon|)^{1/2} }
  |f \overline{v}|
  +
  \textstyle
  \frac{|\epsilon|}{(\lambda_{+}+\sigma^{2}+|\epsilon|)^{1/2}}
  \Re\partial_{j}\{\overline{v}a_{jk} \partial_{k}^{b}v\}-
    \\
  &-
  (N\|\nabla \psi\|^{2}_{L^{\infty}}+1)
  (\lambda_{+}+\sigma^{2}+|\epsilon|)^{1/2}
  \Im\partial_{j}\{(\sgn \epsilon)\overline{v}a_{jk} 
  \partial_{k}^{b}v\}.
\end{split}
\end{equation*}
Then we notice that for any vector valued function $Z=Z(x)$ 
\begin{equation*}
\begin{split}
  (\lambda_{+}+\sigma^{2}+|\epsilon|)^{1/2}
  \nabla \cdot Z
  = &
  \nabla \cdot\{(\lambda_{+}+\sigma^{2}+|\epsilon|)^{1/2} Z\}
  -
  \textstyle
  \frac{\sigma \nabla \sigma \cdot Z}
    {(\lambda_{+}+\sigma^{2}+|\epsilon|)^{1/2}}
    \\
  \le &
  \nabla \cdot\{(\lambda_{+}+\sigma^{2}+|\epsilon|)^{1/2} Z\}
  +|\nabla \sigma \cdot Z|
\end{split}
\end{equation*}
and similarly
\begin{equation*}
\begin{split}
  \textstyle
  \frac{|\epsilon|\nabla \cdot Z}
     {(\lambda_{+}+\sigma^{2}+|\epsilon|)^{1/2}}
  = &
  \textstyle
  \nabla \cdot \{
    \frac{|\epsilon| Z}{(\lambda_{+}+\sigma^{2}+
      |\epsilon|)^{1/2}}\} 
  +
  \frac{|\epsilon| \sigma \nabla \sigma \cdot Z}
     {(\lambda_{+}+\sigma^{2}+|\epsilon|)^{3/2}} 
    \\
  \le &
  \textstyle
  \nabla \cdot \{
    \frac{|\epsilon| Z}{(\lambda_{+}+\sigma^{2}+
    |\epsilon|)^{1/2}}\} 
  +|\nabla \sigma \cdot Z|.
\end{split}
\end{equation*}
Applying these estimates to the previous inequality we get
\begin{equation}\label{eq:ep1}
  |I_{\epsilon}|
  \le 
  (N\|\nabla \psi\|^{2}_{L^{\infty}}+2)
  \Bigl[
  {(\lambda_{+}+\sigma^{2}+|\epsilon|)^{1/2} }
  |f \overline{v}|
  +|a(\nabla^{b} v,v \nabla \sigma)|
  \Bigr]
  +
  \partial_{j}G_{j}
\end{equation}
where
\begin{equation*}
  \textstyle
  G_{j}=
    \frac{|\epsilon|\Re (\overline{v}a_{jk} \partial_{k}^{b}v)}
      {(\lambda_{+}+\sigma^{2}+|\epsilon|)^{1/2}}
    -
    (N\|\nabla \psi\|^{2}_{L^{\infty}}+1)
    (\lambda_{+}+\sigma^{2}+|\epsilon|)^{1/2}
    (\sgn \epsilon)
    \Im (\overline{v}a_{jk} \partial_{k}^{b}v).
\end{equation*}
We then integrate \eqref{eq:ep1} over the set
$\Omega\cap \{|x|\le R\}$. The boundary terms $G_{j}$ vanish
at $\partial \Omega$
in view of the Dirichlet conditions; on the other hand, at the 
remaining part of the boundary $\Omega\cap \{|x|=R\}$ we get the
quantity
\begin{equation*}
  \textstyle
  \int_{\Omega_{=R}}
  \nu_{j}G_{j}dS
\end{equation*}
where $\vec{\nu}=(\nu_{1},\dots,\nu_{n})$ 
is the exterior normal
and $dS$ is the surface measure on the sphere $\{|x|= R\}$.
Since $v\in H^{1}(\Omega)$, and the coefficients
and $\sigma(x)$ are bounded for $x$ large, we get
\begin{equation*}
  \textstyle
  \liminf_{R\to+\infty}
  \int_{\Omega\cap \{|x|=R\}}
  \nu_{j}G_{j}dS=0
\end{equation*}
and hence the boundary term vanishes after
integration over $\Omega$:
\begin{equation*}
  \textstyle
  \int_{\Omega}
  (\partial_{j}G_{j})dx=0.
\end{equation*}
Thus integration of \eqref{eq:ep1} over $\Omega$ gives
\begin{equation}\label{eq:ep2}
  \textstyle
  \int_{\Omega}|I_{\epsilon}|dx
  \le
  (N\|\nabla \psi\|^{2}_{L^{\infty}}+2)
  \int_{\Omega}
  \bigl[
  {(\lambda_{+}+\sigma^{2}+|\epsilon|)^{1/2} }
  |f \overline{v}|
  +|a(\nabla^{b} v,v \nabla \sigma)|
  \bigr]
  dx.
\end{equation}
In order to control the RHS of \eqref{eq:ep2} we need a few more
estimates, which are different for the homogeneous and the
nonhomogeneous cases. We begin with the homogeneous estimate,
under assumption \eqref{eq:assc}. Thus we can take
\begin{equation*}
  \textstyle
  \sigma(x)=
  \frac{C_{-}}{|x|^{1+\delta/2}\vee|x|^{1-\delta/2}}
  \le C_{-}|x|^{-1}
\end{equation*}
so that
\begin{equation}\label{eq:ep3}
  \textstyle
  \int \sigma|f\overline{v}|\le
  C_{-}\|f\|_{\dot Y^{*}}\||x|^{-1}v\|_{\dot Y}
  \le
  C_{-}\|f\|_{\dot Y^{*}}\|v\|_{\dot X}
\end{equation}
(see \eqref{eq:stnorma8}). 
Second, we take the imaginary part of \eqref{eq:morid2}
\begin{equation*}%
  \partial_{j}\Im P_{j}=
  -\epsilon|v|^{2}\phi
  +\Im(f \overline{v})\phi
  +\Im a(\nabla^{b} v,v \nabla \phi),
\end{equation*}
and we choose $\phi$ as follows:
\begin{equation}\label{eq:choicephi}
  \textstyle
  \phi(x)=1 \ \text{if}\ |x|\le R,\quad
  \phi(x)=2-\frac{|x|}{R} \ \text{if}\ R\le|x|\le2R,\quad
  \phi(x)=0 \ \text{if}\ |x|\ge R.
\end{equation}
We note also that the integral of $\partial_{j} P_{j}$ over
$\Omega$ vanishes as above by the Dirichlet boundary conditions.
This gives easily, using \eqref{eq:stnorma5},
\begin{equation}\label{eq:interm}
\begin{split}
  \textstyle
  |\epsilon|
  \int_{\Omega_{\le R}}
  |v|^{2}
  \le &
  \textstyle
  \int_{\Omega_{\le2R}}
  |f \overline{v}|
  +\frac{N}{R}
  \int_{\Omega_{R\le|x|\le 2R}}
  |v||\nabla^{b} v|
    \\
  \le &
  2R\|f\|_{\dot Y^{*}}\|v\|_{\dot X}+
  3NR\|v\|_{\dot X}\|\nabla^{b} v\|_{\dot Y}.
\end{split}
\end{equation}
Dividing by $R$ and taking the sup for $R>0$ gives
\begin{equation}\label{eq:epsest}
  |\epsilon|\|v\|_{\dot Y}^{2}
  \le
  3(1+N)
  \left[\|f\|_{\dot Y^{*}}+\|\nabla^{b} v\|_{\dot Y}\right]
  \|v\|_{\dot X}.
\end{equation}
As a consequence, we can write
\begin{equation*}
  \textstyle
  |\epsilon|^{1/2}\int_{\Omega}|f \overline{v}|
  \le
  \|f\|_{\dot Y^{*}}
  \left[3(1+N)\right]^{1/2}
  \left[\|f\|_{\dot Y^{*}}+\|\nabla^{b} v\|_{\dot Y}\right]^{1/2}
  \|v\|_{\dot X}^{1/2}
\end{equation*}
and this implies
\begin{equation}\label{eq:ep4}
  \textstyle
  |\epsilon|^{1/2}\int_{\Omega}|f \overline{v}|
  \le
  (N+1)^{1/2}
  \left[
    \|f\|_{\dot Y^{*}}^{2}
    +\|f\|_{\dot Y^{*}}\|\nabla^{b} v\|_{\dot Y}
    +\|f\|_{\dot Y^{*}}\|v\|_{\dot X}
  \right].
\end{equation}
If we take instead the real part of \eqref{eq:morid2},
split $c=c_{+}-c_{-}$ into positive and negative 
part, and write $\lambda_{\pm}=\max\{\pm\lambda,0\}$,
we have
\begin{equation*}
  (c_{-}+\lambda_{+})|v|^{2}\phi=
  -\partial_{j}\Re P_{j}
  +a(\nabla^{b}v,\nabla^{b}v)\phi
  +(c_{+}+\lambda_{-})|v|^{2}\phi
  +\Re(f \overline{v})\phi
  -\textstyle \frac12 A \phi|v|^{2}.
\end{equation*}
We choose the same function $\phi$ as in \eqref{eq:choicephi}.
When $\lambda=\lambda_{+}\ge0$,
by assumptions \eqref{eq:assaabove}, \eqref{eq:assader}
we can write
\begin{equation*}
  \textstyle
  -\frac12 A \phi=
  \frac{\overline{a}-\widehat{a}+|x|a_{jk;j}\widehat{x}_{k}}{R|x|}
  \one{R\le|x|\le 2R}
  +\frac{\widehat{a}}{R}(\delta_{|x|=2R}-\delta_{|x|=R})
  \le
  n\frac{N+nC_{a}}{R|x|}\one{R\le|x|\le 2R}
  +\frac{N}{R}\delta_{|x|=2R}
\end{equation*}
and hence
\begin{equation}\label{eq:interm2}
\begin{split}
  \textstyle
  \lambda_{+}|v|^{2}\one{|x|\le R}
  \le
  -\partial_{j}\Re P_{j}
  &
  +(N|\nabla^{b}v|^{2}+c_{+}|v|^{2}+|f \overline{v}|)
  \one{|x|\le 2R}
  \\
  &
  \textstyle
  +n\frac{N+nC_{a}}{R|x|}|v|^{2}\one{R\le|x|\le 2R}
  +\frac{N}{R}|v|^{2}\delta_{|x|=2R}.
\end{split}
\end{equation}
Integrating over $\Omega$, the boundary term disappears 
as usual, and
dividing by $R$ and taking the sup over $R>0$ we get
\begin{equation*}
  \lambda_{+}\|v\|_{\dot Y}^{2}\le
  2N\|\nabla^{b}v\|_{\dot Y}^{2}
  +2\|c_{+}^{1/2}v\|_{\dot Y}^{2}
  +4(N(n+1)+n^{2}C_{a}+1)\|v\|_{\dot X}^{2}
  +\|f\|_{\dot Y^{*}}^{2}
\end{equation*}
(see \eqref{eq:stnorma5}, \eqref{eq:stnorma9}). Using the 
second half of assumption \eqref{eq:assc}
and \eqref{eq:stnorma8}, this gives
\begin{equation}\label{eq:lambdapos}
  \lambda_{+}\|v\|_{\dot Y}^{2}\le
  2N\|\nabla^{b}v\|_{\dot Y}^{2}
  +4(C_{+}^{2}+N(n+1)+n^{2}C_{a}+1)\|v\|_{\dot X}^{2}
  +\|f\|_{\dot Y^{*}}^{2}
\end{equation}
and hence
\begin{equation}\label{eq:ep5}
  \textstyle
  \sqrt{\lambda_{+}}\int_{\Omega}|f \overline{v}|
  \le
  \|f\|_{\dot Y^{*}}
  \left[
    2N\|\nabla^{b}v\|_{\dot Y}^{2}
    +4(C_{+}^{2}+N(n+1)+n^{2}C_{a}+1)\|v\|_{\dot X}^{2}
    +\|f\|_{\dot Y^{*}}^{2}
  \right]^{1/2}
\end{equation}
On the other hand, when $\lambda\le0$ i.e. $\lambda=-\lambda_{-}$,
we rewrite \eqref{eq:morid2} in the form
\begin{equation*}
  (c_{+}+\lambda_{-})|v|^{2}\phi=
  \partial_{j}\Re P_{j}
  -a(\nabla^{b}v,\nabla^{b}v)\
  +(c_{-}+\lambda_{+})|v|^{2}\phi
  -\Re(f \phi\overline{v})\phi
  +\textstyle \frac12 A \phi|v|^{2}
\end{equation*}
and with the same choice of $\phi$ as above we have
(since $\overline{a}\ge \widehat{a}$)
\begin{equation*}
  \textstyle
  \frac12 A \phi=
  -\frac{\overline{a}-\widehat{a}+|x|a_{jk;j}\widehat{x}_{k}}{R|x|}
  \one{R\le|x|\le 2R}
  +\frac{\widehat{a}}{R}(\delta_{|x|=R}-\delta_{|x|=2R})
  \le
  \frac{n^{2}C_{a}}{R|x|}\one{R\le|x|\le 2R}
  +\frac{N}{R}\delta_{|x|=R}
\end{equation*}
which implies ($\lambda_{+}=0$)
\begin{equation}\label{eq:interm3}
  \textstyle
  \lambda_{-}|v|^{2}\one{|x|\le R}
  \le
  \partial_{j}\Re P_{j}
  +(c_{-}|v|^{2}+|f \overline{v}|)
  \one{|x|\le 2R}
  +\frac{n^{2}C_{a}}{R|x|}|v|^{2}\one{R\le|x|\le 2R}
  +\frac{N}{R}|v|^{2}\delta_{|x|=R}.
\end{equation}
Proceeding exactly as above but using the fact that
$c_{-}(x)\le C_{-}^{2}|x|^{-2}$
by \eqref{eq:assc}, we conclude that
\begin{equation}\label{eq:lambdaneg}
  \lambda_{-}\|v\|_{\dot Y}^{2}\le
  2(C_{-}^{2}+N+n^{2}C_{a}+1)\|v\|_{\dot X}^{2}
  +\|f\|_{\dot Y^{*}}^{2}.
\end{equation}
For the remaining term in \eqref{eq:ep2}, we note that
\begin{equation*}
  \textstyle
  \sigma(x)=
  \frac{C_{-}}{|x|^{1+\delta/2}\vee|x|^{1-\delta/2}}
  \quad\implies\quad
  |\nabla\sigma|
  \le
  2C_{-} \cdot
  \begin{cases}
    |x|^{-2+\delta/2} &\text{if $ |x|\le1 $,}\\
    |x|^{-2-\delta/2} &\text{if $ |x|\ge1 $.}
  \end{cases}
\end{equation*}
and by \eqref{eq:stnorma10} we get
\begin{equation}\label{eq:ep6}
  \textstyle
  \int_{\Omega}|a(\nabla^{b} v,v \nabla \sigma)|\le
  36 \delta^{-1}NC_{-}
  \|\nabla^{b}v\|_{\dot Y}\|v\|_{\dot X}.
\end{equation}
We assume now that
\begin{equation}\label{eq:boundpsi}
  \|\nabla \psi\|_{L^{\infty}}\le1
\end{equation}
(the explicit choice of the weight will be done in the
next step) so that,
summing up, we can estimate \eqref{eq:ep2} via
\eqref{eq:ep3}, \eqref{eq:ep4}, \eqref{eq:ep5}, 
\eqref{eq:ep6} to obtain
\begin{equation*}
\begin{split}
  \textstyle
  (N+2)^{-1}
  \int_{\Omega}|I_{\epsilon}|\le
  & \textstyle
    4(C_{-}^{2}+C_{+}^{2}+\frac43 Nn+n^{2}C_{a}+1)^{1/2}
    \|f\|_{\dot Y^{*}}\|v\|_{\dot X}
    \\
    &
    \textstyle
    +\frac52\sqrt{N+1}\|f\|_{\dot Y^{*}}\|\nabla^{b} v\|_{\dot Y}
    +\|f\|_{\dot Y^{*}}^{2}
    +36 \delta^{-1}N C_{-}\|\nabla^{b}v\|_{\dot Y}\|v\|_{\dot X}
\end{split}
\end{equation*}
whence one gets, for any $0<\zeta\le1$, the final
homogeneous estimate of $I_{\epsilon}$
\begin{equation}\label{eq:estIep}
\begin{split}
  \textstyle
  \int_{\Omega}|I_{\epsilon}|
  \le&
  (\zeta+18 \delta^{-1}N(N+2)C_{-})
  [\|\nabla^{b}v\|_{\dot Y}^{2}+\|v\|_{\dot X}^{2}]
  \\
  &+
  8(N+2)^{2}(C_{-}^{2}+C_{+}^{2}+Nn+n^{2}C_{a}+1)\zeta^{-1}
  \|f\|_{\dot Y^{*}}^{2}
\end{split}
\end{equation}
under assumption \eqref{eq:assc}.

We show now how to get a nonhomogeneous estimate 
for $I_{\epsilon}$ under the stronger assumption 
\eqref{eq:asscstrong}, thus we can take now
\begin{equation*}
  \sigma(x)=C_{-}\bra{x}^{-1-\delta/2}.
\end{equation*}
Starting again from \eqref{eq:ep2}, we have,
since $\sigma(x)\le C_{-}|x|^{-1}$,
\begin{equation}\label{eq:ep3bis}
  \textstyle
  \int \sigma|f \overline{v}|
  \le C_{-}\|f\|_{Y^{*}}\||x|^{-1} v\|_{Y}\le
  3C_{-}\|f\|_{Y^{*}}
  (\| v\|_{X}^{2}+\|\nabla^{b}v\|_{Y}^{2})^{1/2}
\end{equation}
by \eqref{eq:stnorma12}. Then we consider again \eqref{eq:interm}
where we estimate as follows
\begin{equation*}
  \textstyle
  |\epsilon|\int_{\Omega_{\le R}}|v|^{2}\le
  \bra{2R}\|v\|_{X}\|f\|_{Y^{*}}+
  \frac{N}{R}\cdot 3 \bra{R}^{2}\|v\|_{X}\|\nabla^{b}v\|_{Y}
\end{equation*}
by \eqref{eq:stnorma5b}. Dividing by $R$, taking
the sup over $R>1$ and recalling \eqref{eq:equivnonhom}
we get
\begin{equation}\label{eq:epsestbis}
  |\epsilon|\|v\|_{Y}^{2}\le
  \sqrt{5}\|v\|_{X}\|f\|_{Y^{*}}+
  6N\|v\|_{X}\|\nabla^{b}v\|_{Y}
\end{equation}
Thus
\begin{equation*}
  \textstyle
  |\epsilon|^{1/2}\int_{\Omega}|f \overline{v}|
  \le
  \|f\|_{ Y^{*}}
  \left[  \sqrt{5}
  \|f\|_{ Y^{*}}+6N\|\nabla^{b} v\|_{ Y}\right]^{1/2}
  \|v\|_{ X}^{1/2}
\end{equation*}
and this implies
\begin{equation}\label{eq:ep4bis}
  \textstyle
  |\epsilon|^{1/2}\int_{\Omega}|f \overline{v}|
  \le
  2(N+1)^{1/2}
  \left[
    \|f\|_{Y^{*}}^{2}
    +\|f\|_{Y^{*}}\|\nabla^{b} v\|_{Y}
    +\|f\|_{Y^{*}}\|v\|_{X}
  \right].
\end{equation}
Next, in the case $\lambda=\lambda_{+}\ge0$
we integrate \eqref{eq:interm2} over $\Omega$ and
we take the sup over $R>1$ to get
\begin{equation*}
  \lambda_{+}\|v\|_{Y}^{2}\le
  2\sqrt{2}N\|\nabla^{b}v\|_{Y}^{2}+
  2\sqrt{2}\|c_{+}^{1/2}v\|_{Y}^{2}+
  (3n(N+nC_{a})+2N+2)\|v\|_{X}^{2}+
  \|f\|^{2}_{Y^{*}}
\end{equation*}
where we used \eqref{eq:stnorma5b} and \eqref{eq:stnorma6}.
Since $c_{+}(x)\le C_{+}|x|^{-2}$, by \eqref{eq:stnorma12}
we obtain
\begin{equation}\label{eq:lambdaposbis}
  \lambda_{+}\|v\|_{Y}^{2}\le
  3(N+6C_{+}^{2})\|\nabla^{b}v\|_{Y}^{2}
  +3(N(n+1)+n^{2}C_{a}+3C_{+}^{2})\|v\|_{X}^{2}
  +\|f\|_{Y^{*}}^{2}
\end{equation}
and this implies
\begin{equation}\label{eq:ep5bis}
\begin{split}
  \sqrt{\lambda_{+}} & 
  \textstyle
  \int_{\Omega}|f \overline{v}|dx
  \le
    \\
  &\le\|f\|_{Y^{*}}
  \left[
    3(N+6C_{+}^{2})\|\nabla^{b}v\|_{Y}^{2}
  +3(N(n+1)+n^{2}C_{a}+3C_{+}^{2})\|v\|_{X}^{2}
  +\|f\|_{Y^{*}}^{2}
  \right]^{1/2}.
\end{split}
\end{equation}
On the other hand, when $\lambda=-\lambda_{-}\le 0$,
we integrate \eqref{eq:interm3} over $\Omega$,
divide by $R$ and take the sup over $R>1$ to obtain,
using \eqref{eq:stnorma5b}, \eqref{eq:stnorma6},
\begin{equation*}
  \lambda_{-}\|v\|_{Y}^{2}\le
  2 \sqrt{2}\|c_{-}^{1/2}v\|_{Y}^{2}
  +(2(N+1)+3n^{2}C_{a})\|v\|_{X}^{2}
  +\|f\|_{Y^{*}}^{2}
\end{equation*}
and since $c_{-}(x)\le C_{-}|x|^{-1}$, by \eqref{eq:stnorma12}
we get
\begin{equation}\label{eq:lambdanegbis}
  \lambda_{-}\|v\|_{Y}^{2}\le
  18C_{-}^{2}\|\nabla^{b}v\|_{Y}^{2}
  +3(N+n^{2}C_{a}+6C_{-}^{2}+1)\|v\|_{X}^{2}
  +\|f\|_{Y^{*}}^{2}
\end{equation}
To estimate last term in \eqref{eq:ep2} 
we notice that
$|\nabla \sigma|\le 2C_{-}\bra{x}^{-2-\delta/2}$ 
and hence
\begin{equation}\label{eq:ep6bis}
  \textstyle
  \int_{\Omega}|a(\nabla^{b} v,v \nabla \sigma)|\le
  2NC_{-}\int_{\Omega}
     \frac{|\nabla^{b}v||v|}{\bra{x}^{2+\delta/2}}dx\le
  18 \delta^{-1}NC_{-}
  (\|\nabla^{b}v\|_{Y}^{2}+\|v\|_{X}^{2})
\end{equation}
by \eqref{eq:stnorma13}. Summing up we get in a few steps
(using again \eqref{eq:boundpsi})
\begin{equation*}
\begin{split}
  \textstyle
  (N+2)^{-1}
  \int_{\Omega}|I_{\epsilon}|
  \le &
  9(C_{-}^{2}+4C_{+}^{2}+Nn+n^{2}C_{a}+1)^{1/2}
  \|f\|_{Y^{*}}[\|v\|_{X}+\|\nabla^{b}v\|_{Y}]
  \\
  &+3(N+1)\|f\|^{2}_{Y^{*}}
  +18 \delta^{-1}NC_{-}
  (\|\nabla^{b}v\|_{Y}^{2}+\|v\|_{X}^{2})
\end{split}
\end{equation*}
and hence, for every $0<\zeta\le1$, we obtain
\begin{equation}\label{eq:estIepbis}
\begin{split}
  \textstyle
  \int_{\Omega}|I_{\epsilon}|
  \le &
  (\zeta+ 18 \delta^{-1}N(N+2)C_{-})
  [\|\nabla^{b}v\|^{2}_{Y}+\|v\|^{2}_{X}]
\\
  &+
  44(N+2)^{2}(C_{-}^{2}+4C_{+}^{2}+Nn+n^{2}C_{a}+1)
  \zeta^{-1}\|f\|_{Y^{*}}^{2}
\end{split}
\end{equation}
under assumption \eqref{eq:asscstrong}.

\subsection{Choice of the weight \texorpdfstring{$\psi$}{}}
\label{sub:choice_of_the_weight_y_}

Our choice of the weight function $\psi$ in \eqref{eq:morid}
is inspired by
\cite{DAnconaRacke12-a},
\cite{BarceloFanelliRuiz13-a}  
(see also 
\cite{DAnconaFanelli09-a},
\cite{Fanelli09-b}). Define
\begin{equation}\label{psi0}
  \textstyle
  \psi_0(r)=\int_0^r\psi'(s)ds
\end{equation}
where
\begin{equation*}
  \psi'_1(r)=
  \begin{cases}
  \frac{n-1}{2n}r,
  & r\leq1
  \\
  \frac{1}{2}-\frac{1}{2nr^{n-1}},
  & r>1.
  \end{cases}
\end{equation*}
Then $\psi$ is the
radial function, depending on a scaling parameter $R>0$,
\begin{equation*}
  \textstyle
  \psi(|x|)\equiv\psi_R(|x|):=R\psi_1\left(\frac{|x|}{R}\right).
\end{equation*}
Here and in the following, with a slight abuse, we shall use
the same letter $\psi$ to denote a function $\psi(r)$
defined for $r\in \mathbb{R}^{+}$ and
the radial function $\psi(x)=\psi(|x|)$ 
defined on $\mathbb{R}^{n}$. We compute the first radial
derivatives $\psi^{(j)}(r)=(\frac{x}{|x|}\cdot \nabla)^{j}\psi(x)$
for $|x|>0$:
\begin{equation}\label{eq:psi1}
  \psi'(x)=
  \begin{cases}
  \frac{n-1}{2n}\cdot\frac{|x|}{R},& |x|\leq R
  \\
  \frac{1}{2}-\frac{R^{n-1}}{2n|x|^{n-1}},
  & |x|>R
  \end{cases}
\end{equation}
which can be equivalently written as
\begin{equation*}
  \textstyle
  \psi'(x)=
  \frac{|x|}{2nR}
  \left[
    n \frac{R}{R\vee|x|}-(\frac{R}{R\vee|x|})^{n}
  \right]
\end{equation*}
and implies in particular that the assumption
\eqref{eq:boundpsi} used in the previous
step is satisfied, and actually
\begin{equation}\label{eq:bddpsip}
  \textstyle
  0\le \psi'\le \frac{1}{2}.
\end{equation}
Then we have
\begin{equation}\label{eq:psi2}
  \psi''(x)=
  \textstyle
  \frac{n-1}{2n}\cdot\frac{R^{n-1}}{(R\vee|x|)^n}
  =
  \frac{n-1}{2n}\cdot
  \begin{cases}
  \frac{1}{R} & |x|\leq R
  \\
  \frac{R^{n-1}}{|x|^n} & |x|>R,
  \end{cases}
\end{equation}
\begin{equation}\label{eq:psi3}
  \psi'''(x)=
  \textstyle
  -\frac{n-1}{2}\frac{R^{n-1}}{|x|^{n+1}}\one{|x|\ge R}
\end{equation}
\begin{equation}\label{eq:psi4}
  \psi^{IV}(x)=
  \textstyle
  \frac{n^{2}-1}{2}\cdot \frac{R^{n-1}}{|x|^{n+2}}\one{|x|\ge R}
  -
  \frac{n-1}{2}\frac{1}{R^{2}}\delta_{|x|= R}
  .
\end{equation}
%
%
Note in particular that
\begin{equation}\label{eq:diffpsi}
  \psi''-\frac{\psi'}{|x|}=
  \begin{cases}
    0
     &|x|\le R\\
    -\frac{1}{2|x|}
    \left(1-\frac{R^{n-1}}{|x|^{n-1}}\right)
     &|x|>R.
  \end{cases}
\end{equation}
Moreover, we choose $\phi=0$ and we see
that (see \eqref{eq:apsigen})
\begin{equation}\label{eq:Apsiphi}
  \textstyle
  A \psi+\phi = A \psi=
  \widehat{a}\psi''
  +
  \frac{\overline{a}-\widehat{a}}{|x|}
  \psi'
  +
  a_{\ell m;\ell}\widehat{x}_{m}\psi'.
\end{equation}
is continuous and piecewise Lipschitz.

\subsection{Estimate of the terms in \texorpdfstring{$|v|^{2}$}{}}
\label{sub:v2}

Since $\phi\equiv0$, these terms reduce to
\begin{equation}\label{eq:termsv2}
  I_{|v|^{2}}=
  -\frac12 A^{2}\psi|v|^{2}
  -a(\nabla \psi,\nabla c)|v|^{2}.
\end{equation}
First of all we 
need to compute the quantity $A(A \psi+\phi)\equiv A^{2}\psi$.
Using the identity
\begin{equation*}
  A(uv)=(Au)v+u(Av)+2a(\nabla u,\nabla v)
\end{equation*}
we can write
\begin{equation*}
  A^{2}\psi=I+II+III+A(a_{\ell m;\ell}\widehat{x}_{m}\psi')
\end{equation*}
where
\begin{equation*}
  I=\widehat{a}\cdot A {\psi''}+
  (\overline{a}-\widehat{a})A\left(\frac{\psi'}{|x|}\right),
\end{equation*}
\begin{equation*}
  II=A \widehat{a}\cdot {\psi''}+
  A(\overline{a}-\widehat{a})\cdot \frac{\psi'}{|x|}
\end{equation*}
and
\begin{equation*}
  \textstyle
  III=2a(\nabla \widehat{a},\nabla {\psi''})+
  2a(\nabla(\overline{a}-\widehat{a}),\nabla \frac{\psi'}{|x|}).
\end{equation*}
We separate the terms which do not contain
derivatives of the coefficients $a_{jk}$ from the others.
We have $I=I_{1}+I_{2}$ with
\begin{equation*}
  \textstyle
  I_{1}=
  \widehat{a}^{2}{\psi^{IV}}+
  \widehat{a}(\overline{a}-\widehat{a})
  \frac{{\psi'''}+\psi'''}{|x|}+
  \frac{(\overline{a}-\widehat{a})(\overline{a}-3 
  \widehat{a})}{|x|^{2}}
  \left(
  \psi''-\frac{\psi'}{|x|}
  \right),
\end{equation*}
\begin{equation*}
  \textstyle
  I_{2}=
  \widehat{a}a_{\ell m;\ell}\widehat{x}_{m}{\psi'''}+
  (\overline{a}-\widehat{a})
  a_{jk;j}\widehat{x}_{k}
  \left(
  \frac{\psi''}{|x|}-\frac{\psi'}{|x|^{2}}
  \right)
\end{equation*}
where
\begin{equation}\label{eq:psiti3}
  \psi'''
  \textstyle
  =
  -\frac{n-1}{2}
  \frac{R^{n-1}}{|x|^{n+1}}\one{|x|\ge R},
\end{equation}
\begin{equation}\label{eq:psiti4}
  \textstyle
  {\psi^{IV}}
  =
  \frac{n^{2}-1}{2}
  \frac{R^{n-1}}{|x|^{n+2}}\one{|x|\ge R}
  +
  \frac{n-1}{2}\frac{1}{R^{2}}\delta_{|x|=R}.
\end{equation}
In a similar way, $II=II_{1}+II_{2}$ with
\begin{equation*}
  \textstyle
  II_{1}=
  \frac{2}{|x|^{2}}
  [
  a_{\ell m}a_{\ell m}
  -4(|a \widehat{x}|^{2}-\widehat{a}^{2})
  -\overline{a}\widehat{a}
  ]
  \left(
  {\psi''}-\frac{\psi'}{|x|}
  \right),
\end{equation*}
\begin{equation*}
  \textstyle
  II_{2}=
  [
  \partial_{j}(a_{jk}a_{\ell m;k}\widehat{x}_{\ell}
  \widehat{x}_{m})
  +\partial_{j}(a_{jk}a_{\ell m})\partial_{k}(\widehat{x}_{\ell}
      \widehat{x}_{m})
  ]
  \left(
  {\psi''}-\frac{\psi'}{|x|}
  \right)
  +
  (A \overline{a})\frac{\psi'}{|x|}
\end{equation*}
and $III=III_{1}+III_{2}$ with
\begin{equation*}
  \textstyle
  III_{1}=
  \frac{4}{|x|}[|a \widehat{x}|^{2}-\widehat{a}^{2}]
  \left(
  {\psi'''}
  -\frac{\psi''}{|x|}
  +\frac{\psi'}{|x|^{2}}
  \right),
\end{equation*}
\begin{equation*}
  \textstyle
  III_{2}=
  2a_{jk}a_{\ell m;k}\widehat{x}_{\ell}\widehat{x}_{m}
  \widehat{x}_{j}
  \left({\psi'''}-\frac{\psi''}{|x|}\right)
  +2a(\nabla \overline{a},\nabla \frac{\psi'}{|x|}).
\end{equation*}
Collecting all the terms  we get
\begin{equation}\label{eq:splitting}
  A^{2}\psi=S(x)+R(x)
\end{equation}
where
\begin{equation*}
\begin{split}
  S(x)=&
  \textstyle
  \widehat{a}^{2}{\psi^{IV}}+
  2\widehat{a}(\overline{a}-\widehat{a})
  \frac{\psi'''}{|x|}+
  \frac{(\overline{a}-\widehat{a})(\overline{a}-3 
  \widehat{a})}{|x|^{2}}
  \left(
  \psi''-\frac{\psi'}{|x|}
  \right)+
    \\
  &+
  \textstyle
  \frac{2}{|x|^{2}}
  [
  a_{\ell m}a_{\ell m}
  -\overline{a}\widehat{a}
  -4(|a \widehat{x}|^{2}-\widehat{a}^{2})
  ]
  \left(
  {\psi''}-\frac{\psi'}{|x|}
  \right)+
    \\
  &
  \textstyle
  +\frac{4}{|x|}[|a \widehat{x}|^{2}-\widehat{a}^{2}]
  \left(
  {\psi'''}
  -\frac{\psi''}{|x|}
  +\frac{\psi'}{|x|^{2}}
  \right)
\end{split}
\end{equation*}
and
\begin{equation*}
  \textstyle
\begin{split}
  R(x)=
  &
  \widehat{a}a_{\ell m;\ell}\widehat{x}_{m}{\psi'''}+
  (\overline{a}-\widehat{a})
  a_{jk;j}\widehat{x}_{k}
  \textstyle
  \left(
  \frac{\psi''}{|x|}-\frac{\psi'}{|x|^{2}}
  \right)+
    \\
  &+
    [
  \partial_{j}(a_{jk}a_{\ell m;k}\widehat{x}_{\ell}
  \widehat{x}_{m})
  +\partial_{j}(a_{jk}a_{\ell m})\partial_{k}(\widehat{x}_{\ell}
      \widehat{x}_{m})
  ]
  \textstyle
  \left(
  {\psi''}-\frac{\psi'}{|x|}
  \right)
  +
  (A \overline{a})\frac{\psi'}{|x|}+
    \\
  &+
    2a_{jk}a_{\ell m;k}\widehat{x}_{\ell}\widehat{x}_{m}
    \widehat{x}_{j}
    \textstyle
    \left({\psi'''}-\frac{\psi''}{|x|}\right)
    +2a(\nabla \overline{a},\nabla \frac{\psi'}{|x|})+
    \\ 
  &+
  \textstyle
  A(a_{\ell m;\ell}\widehat{x}_{m}\psi').
\end{split}
\end{equation*}
The remainder $R(x)$ can be estimated as follows:
recalling that, by definition of $\psi$,
\begin{equation*}
  \textstyle
  |\psi'|\le \frac{|x|}{2 (R\vee|x|)},\qquad
  |\psi''|\le \frac{n-1}{2n (R\vee|x|)},\qquad
  |\psi'''|\le \frac{n-1}{2 (R\vee|x|)|x|}
\end{equation*}
and the assumptions \eqref{eq:assader},
then after a long but elementary computation
we find ($n\ge3$)
\begin{equation}\label{eq:estR}
  |R(x)|\le
  \frac{12 n C_{a}(N+C_{a})}{|x|\bra{x}^{1+\delta}(R\vee|x|)}.
\end{equation}

We focus now on the main term $S(x)$,
which can be written
\begin{equation}\label{eq:Sx2}
\begin{split}
  S(x)=
  &
  \textstyle
  \widehat{a}^{2}\psi^{IV}+
  [2 \overline{a}\widehat{a}-6 \widehat{a}^{2}+4|a 
  \widehat{x}|^{2}]
     \frac{\psi'''}{|x|}+
    \\
  &+
  \textstyle
  [2a_{\ell m}a_{\ell m}
     +\overline{a}^{2}
     -6 \overline{a}\widehat{a}
     +15 \widehat{a}^{2}
     -12 |a \widehat{x}|^{2}]
     \left(\frac{\psi''}{|x|^{2}}-\frac{\psi'}{|x|^{3}}\right)
\end{split}
\end{equation}
With our choice of the weight $\psi$
we have in the region $|x|\le R$
\begin{equation}\label{eq:SlessR}
  \textstyle
  S(x)=
  -\frac{n-1}{2}\widehat{a}^{2}
  \frac{1}{R^{2}}\delta_{|x|=R}
\end{equation}
while in the region $|x|>R$
\begin{equation}\label{eq:SgreR}
\begin{split}
  S(x)=
  &
  \textstyle
  (n-1)
  \left[
    \frac{n+3}{2}\widehat{a}-\overline{a}
  \right]
  \widehat{a}
  \frac{R^{n-1}}{|x|^{n+2}}
  -
  2(n-1)
  [|a \widehat{x}|^{2}-\widehat{a}^{2}]
  \frac{R^{n-1}}{|x|^{n+2}}-
    \\
  &
  \textstyle
  -
  [
    2a_{\ell m}a_{\ell m}+\overline{a}^{2}
    -6 \overline{a}\widehat{a}+15 \widehat{a}^{2}
    -12|a \widehat{x}|^{2}
  ]
  \bigl(
  1-(\frac{R}{|x|})^{n-1}
  \bigr)
  \frac{1}{2|x|^{3}}.
\end{split}
\end{equation}
Note that $a_{\ell m}a_{\ell m}$ is the square of the
Hilbert-Schmidt norm of the matrix $a(x)$.

To proceed, we handle the cases of Theorems \ref{the:1}
and \ref{the:2} separately.
For Theorem \ref{the:1}, 
we deduce from assumption \eqref{eq:assaabove}
\begin{equation*}
  nN\ge\overline{a}\ge n\nu,\qquad
  N\ge|a \widehat{x}|\ge\widehat{a}\ge\nu,\qquad
  a_{\ell m}a_{\ell m}\ge n\nu^{2},
\end{equation*}
so that
\begin{equation}\label{eq:nge4S1}
  \textstyle
  S(x)\le -\frac{n-1}{2}\nu\widehat{a}
  \frac{1}{R^{2}}\delta_{|x|=R}
  \qquad\text{for}\ |x|\le R
\end{equation}
while, recalling also assumption \eqref{eq:asscaseA},
we obtain 
\begin{equation}\label{eq:nge4S2}
  S(x)\le
  \textstyle
  (n-1)
  \left[\frac{n+3}{2}N-n\nu\right]\widehat{a}
  \frac{R^{n-1}}{|x|^{n+2}}
  \qquad\text{for}\ |x|\ge R.
\end{equation}

On the other hand, for Theorem \ref{the:2} with $n=3$,
writing $a(x)=I+q(x)$ i.e.
$q_{\ell m}:= a_{\ell m}-\delta_{\ell m}$ we have,
with the usual notations
$\widehat{q}=q_{\ell m}\widehat{x}_{\ell}\widehat{x}_{m}$ and
$\overline{q}=q_{\ell\ell}$,
\begin{equation*}
  a_{\ell m}a_{\ell m}=
  \delta_{\ell m}\delta_{\ell m}+
  2 \delta_{\ell m}q_{\ell m}
     +q_{\ell m}q_{\ell m}
  =
  3+2 \overline{q}+q_{\ell m}q_{\ell m}
\end{equation*}
 and also
\begin{equation*}
  \widehat{a}=1+\widehat{q},\qquad
  \overline{a}=3+\overline{q},\qquad
  |a \widehat{x}|^{2}=
  1+2 \widehat{q}+|q \widehat{x}|^{2}.
\end{equation*}
Note that $|q|=|a(x)-I|\le C_{I}\bra{x}^{-\delta}<1$
by assumption \eqref{eq:assperturb}, which implies
\begin{equation*}
  |\overline{q}|\le 3C_{I}\bra{x}^{-\delta},\qquad
  |\widehat{q}|\le C_{I}\bra{x}^{-\delta},\qquad
  |q \widehat{x}|\le C_{I}\bra{x}^{-\delta}
\end{equation*}
so that
\begin{equation*}
\begin{split}
  2a_{\ell m}a_{\ell m}+\overline{a}^{2}
  -6 \overline{a}\widehat{a}+15 \widehat{a}^{2}
  -12|a \widehat{x}|^{2}
  = &
  4 \overline{q}-12 \widehat{q}
  +2q_{\ell m}q_{\ell m}+\overline{q}^{2}
  -6 \overline{q}\widehat{q}
  +15 \widehat{q}^{2}
  -12|q \widehat{x}|^{2}
  \\
  \ge &
  4 \overline{q}-12 \widehat{q}
  -6 \overline{q}\widehat{q}
  -12|q \widehat{x}|^{2}
  \ge
  -46 C_{I}\bra{x}^{-\delta}.
\end{split}
\end{equation*}
We have also $1-C_{I}\le \widehat{a}\le 1+C_{I}$ so that
($n=3$)
\begin{equation*}
  \textstyle
  -\frac{n-1}{2}\widehat{a}^{2}\le 
  -(1-C_{I})^{2},
  \quad
  \textstyle
  \left(\frac{n+3}{2}\widehat{a}-\overline{a}\right)\widehat{a}
  \le 6  C_{I}(1+C_{I})
  < 12 C_{I}
\end{equation*}
Thus under the assumptions of Theorem \ref{the:2}
we obtain the estimates
\begin{equation}\label{eq:neq3S1}
  \textstyle
  S(x)\le -
  (1-C_{I})
  \frac{1}{R^{2}}\delta_{|x|=R}
  \quad\text{for}\ |x|\le R
\end{equation}
and
\begin{equation}\label{eq:neq3S2}
  \textstyle
  S(x)\le 
  24 C_{I}
  \left[
    \frac{R^{2}}{|x|^{5}}+
    \frac{1}{|x|^{3}\langle x\rangle^{\delta}}
  \right]
  \quad\text{for}\ |x|> R.
\end{equation}

We are now ready to estimate the integral
\begin{equation*}
  \textstyle
  -\int_{\Omega}A(A \psi+\phi)|v|^{2}dx
  \equiv 
  -\int_{\Omega}A^{2}\psi|v|^{2}dx=I+II
\end{equation*}
where
\begin{equation*}
  \textstyle
  I=-\int_{\Omega} S(x)|v|^{2}dx,\qquad
  II=-\int_{\Omega} R(x)|v|^{2}dx.
\end{equation*}
In Theorem \ref{the:1}
by \eqref{eq:estR} and \eqref{eq:stnorma1} we have immediately
for any $R>0$
\begin{equation}\label{eq:Rv2}
  \textstyle
  II
  \ge
  -
  24 n \delta^{-1} C_{a}(N+C_{a})
  \|v\|_{\dot X}^{2}.
\end{equation}
In Theorem \ref{the:2} we use a different estimate, which is
valid for $R>1$ only: by \eqref{eq:estR} and
\eqref{eq:stnorma11} we have
\begin{equation}\label{eq:Rv2n3}
  R>1 \quad\implies\quad
  II\ge
  -324\delta^{-1}
  C_{a}(N+C_{a})
  \left[
  \|v\|_{X}^{2}+\|\nabla^{b}v\|^{2}_{L^{2}(\Omega_{\le1})}
  \right].
\end{equation}

Concerning the $S(x)$ term, in Theorem \ref{the:1} we have
by \eqref{eq:nge4S1}, \eqref{eq:nge4S2} 
\begin{equation}\label{eq:Sv2}
  \textstyle
  I
  \ge
  \frac{n-1}{2}\nu
  \frac{1}{R^{2}}\int_{\Omega_{=R}}\widehat{a}|v|^{2}dS
  -
  \left[\frac{n+3}{2}N-n\nu\right]
  (n-1)
  \int_{\Omega_{\ge R}}
  \widehat{a}
  \frac{R^{n-1}}{|x|^{n+2}}
\end{equation}
for all $R>0$. If $\frac{n+3}{2}N\ge n\nu$, we can use
\eqref{eq:stnorma2} to get
\begin{equation*}
  \textstyle
  I\ge   \frac{n-1}{2}\nu
  \frac{1}{R^{2}}\int_{\Omega_{=R}}\widehat{a}|v|^{2}dS
  -
  \left[\frac{n+3}{2}N-n\nu\right]
  \|\widehat{a}^{1/2}v\|_{\dot X}^{2}
\end{equation*}
and taking the sup over $R>0$ we conclude
\begin{equation}\label{eq:estSA}
  \sup_{R>0}
  I\ge K_{0} \nu\|\widehat{a}^{1/2}v\|_{\dot X}^{2}
\end{equation}
provided we define $K_{0}$ as
\begin{equation}\label{eq:defK}
  \textstyle
  K_{0}:=
  \frac{n-1}{2}-\frac{n+3}{2}\frac{N}{\nu}+n>0;
\end{equation}
notice that the condition
$K_{0}>0$ is equivalent to assumption \eqref{eq:assratio}.
On the other hand, if $\frac{n+3}{2}N\le n\nu$
(which implies $K_{0}\le \frac{n-1}{2}$)
from \eqref{eq:Sv2} we have directly
\begin{equation*}
  \textstyle
  I\ge   \frac{n-1}{2}\nu
  \frac{1}{R^{2}}\int_{\Omega_{=R}}\widehat{a}|v|^{2}dS
  \ge K_{0}\nu
  \frac{1}{R^{2}}\int_{\Omega_{=R}}\widehat{a}|v|^{2}dS
\end{equation*}
and taking the sup over $R>0$ we obtain again
\eqref{eq:estSA}.

In Theorem \ref{the:2}, using \eqref{eq:stnorma2} and 
\eqref{eq:stnorma3} in \eqref{eq:neq3S1}, \eqref{eq:neq3S2},
we have for all $R>1$
\begin{equation}\label{eq:Sv2n3}
  \textstyle
  R>1 \quad\implies\quad
  I
  \ge
  (1-C_{I})
  \frac{1}{R^{2}}\int_{\Omega_{=R}}|v|^{2}dS
  -72 C_{I} \delta^{-1}
  \|v\|^{2}_{X}.
\end{equation}

It remains to consider the second term in \eqref{eq:termsv2};
in Theorem \ref{the:1} we have
\begin{equation}\label{eq:repuls}
  \textstyle
  -a(\nabla \psi,\nabla c)|v|^{2}=
  -a(\widehat{x},\nabla c)\psi'|v|^{2}
  \ge 
  -\frac{C_{c}}{|x|^{2}\langle x\rangle^{1+\delta}}
  \psi'|v|^{2}
\end{equation}
thanks to assumption \eqref{eq:assac}.
Since $0<\psi'<1/2$, by estimate \eqref{eq:stnorma1} we obtain
\begin{equation}\label{eq:estc}
  \textstyle
  -\int_{\Omega}a(\nabla \psi,\nabla c)|v|^{2}
  \ge 
  -C_{c}\delta^{-1}\|v\|_{\dot X}^{2}
\end{equation}
and taking into account \eqref{eq:estSA}, \eqref{eq:Rv2}
and the inequality $\widehat{a}\ge\nu$, we obtain
\begin{equation}\label{eq:estIv2}
  \textstyle
  \sup_{R>0}\int_{\Omega}I_{|v|^{2}}dx\ge
  \left[
  K_{0}\nu^{2}-(24 n C_{a}(N+C_{a})+C_{c})\delta^{-1}
  \right]
  \|v\|_{\dot X}^{2}.
\end{equation}
In Theorem \ref{the:2} we use the stronger assumption
\eqref{eq:assacstrong}
and \eqref{eq:stnorma11}, to obtain
\begin{equation}\label{eq:estcbis}
  \textstyle
  -\int_{\Omega}a(\nabla \psi,\nabla c)|v|^{2}
  \ge -8 \delta^{-1} C_{c}\|v\|_{X}^{2}
  -
  9 C_{c}\|\nabla^{b}v\|^{2}_{L^{2}(\Omega_{\le1})}.
\end{equation}
Putting together\eqref{eq:Rv2n3}, \eqref{eq:Sv2n3}
and \eqref{eq:estcbis} we obtain
\begin{equation*}%
\begin{split}
  \textstyle
  \int_{\Omega}I_{|v|^{2}}dx\ge
  &
  \textstyle
  (1-C_{I})
  \frac{1}{R^{2}}\int_{\Omega_{=R}}|v|^{2}dS
  -(72C_{I}\delta^{-1} +8 \delta^{-1}C_{c})\|v\|_{X}^{2}
  -9 C_{c}\|\nabla^{b}v\|^{2}_{L^{2}(\Omega_{\le1})}
  \\
    &
  -324 \delta^{-1}
  C_{a}(N+C_{a})
  \left[
  \|v\|_{X}^{2}+\|\nabla^{b}v\|^{2}_{L^{2}(\Omega_{\le1})}
  \right]
\end{split}
\end{equation*}
whence, noticing that 
$\|w\|_{L^{2}(\Omega_{\le1})}\le \sqrt{2}\|w\|_{Y}$,
we have for all $R>1$
\begin{equation}\label{eq:estIv2bis}
\begin{split}
  \textstyle
  \int_{\Omega}I_{|v|^{2}}dx\ge
  (1-C_{I})
  \frac{1}{R^{2}}\int_{\Omega_{=R}}|v|^{2}dS
  &
  -8 \delta^{-1}
  [C_{c}+9C_{I}+41 C_{a}(N+C_{a})]
  \|v\|^{2}_{X}
  \\
  &
  -13 \delta^{-1}
  [C_{c}+36 C_{a}(N+C_{a})]
  \|\nabla^{b}v\|^{2}_{Y}.
\end{split}
\end{equation}

\subsection{Estimate of the terms in \texorpdfstring{$|\nabla^{b}v|^{2}$}{}}
\label{sub:nablav}

We consider now the terms in \eqref{eq:morid} which are quadratic
in $\nabla^{b} v$: since $\phi=0$ they reduce to
\begin{equation*}
  I_{|\nabla v|^{2}}=
  \alpha_{\ell m}\cdot
  \Re(\partial^{b}_{\ell}v\overline{\partial^{b}_{m}v})
\end{equation*}
with
\begin{equation*}
  \alpha_{\ell m}=2a_{jm}\partial_{j}(a_{\ell k}
       \partial_{k}\psi)
        -a_{jk}a_{\ell m;j}\partial_{k}\psi .
\end{equation*}
We split the coefficient as
\begin{equation*}
  \alpha_{\ell m}=
  s_{\ell m}(x)+r_{\ell m}(x)
\end{equation*}
where the remainder $r_{\ell m}$ gathers the terms containing
derivatives of the $a_{jk}$.
Since the weight $\psi$ is radial we have
\begin{equation*}
  \textstyle
  s_{\ell m}(x)=
  2 a_{jm} a_{\ell k} \widehat{x}_{j}\widehat{x}_{k}
  \left(
    \psi''-\frac{\psi'}{|x|}
  \right)
  +
  2 a_{jm}a_{j\ell}
  \frac{\psi'}{|x|}
\end{equation*}
while
\begin{equation*}
  r_{\ell m}(x)=
  [2a_{jm}a_{\ell k;j}-a_{jk}a_{\ell m;j}]\widehat{x}_{k}\psi'.
\end{equation*}
We estimate directly
\begin{equation*}%
  |r_{\ell m}(x)
  \Re(\partial^{b}_{\ell}v\overline{\partial^{b}_{m}v})
  |\le 
  3|a||a'||\nabla^{b}v|^{2}
\end{equation*}
and by assumption \eqref{eq:assader}
we obtain
\begin{equation*}
  |r_{\ell m}(x)
  \Re(\partial^{b}_{\ell}v\overline{\partial^{b}_{m}v})
  |\le 3N C_{a}\langle x\rangle^{-1-\delta}|\nabla^{b}v|^{2}.
\end{equation*}
Then integration on $\Omega$ gives,
using \eqref{eq:stnorma4},
\begin{equation}\label{eq:estRD}
  \textstyle
  \int_{\Omega}
  |r_{\ell m}
  \Re(\partial^{b}_{\ell}v\overline{\partial^{b}_{m}v})|dx
  \le
  24N C_{a}\delta^{-1}\|\nabla^{b}v\|^{2}_{Y}.
\end{equation}
Concerning $s_{\ell m}$, 
in the region $|x|>R$ we have
\begin{equation*}
  \textstyle
  s_{\ell m}(x)=
  [a_{jm}a_{j\ell}- a_{jm} a_{\ell k} \widehat{x}_{j}
  \widehat{x}_{k}]
    \frac{1}{|x|}
  +\frac{R^{n-1}}{|x|^{n}}
  a_{jm} a_{\ell k} \widehat{x}_{j}\widehat{x}_{k}
  -
  a_{jm}a_{j\ell}\frac{R^{n-1}}{n|x|^{n}}
\end{equation*}
so that, in the sense of positivity of matrices,
\begin{equation*}
  \textstyle
  s_{\ell m}(x)\ge
  [a_{jm}a_{j\ell}- a_{jm} a_{\ell k} \widehat{x}_{j}
  \widehat{x}_{k}]
    \frac{n-1}{n|x|}\ge0
    \quad\text{for $|x|>R$}
\end{equation*}
(indeed, one has
$a_{jm}a_{j\ell}\ge a_{jm}a_{\ell k}\widehat{x}_{j}
\widehat{x}_{k}$
as matrices); on the other hand,
in the region $|x|\le R$ we have
\begin{equation*}
  \textstyle
  s_{\ell m}(x)=
  a_{jm}a_{j\ell}
    \frac{n-1}{nR}
  \quad\text{for $|x|\le R$}.
\end{equation*}
Thus, by the assumption $a(x)\ge\nu I$, one has for all $x$
\begin{equation}\label{eq:estder}
  \textstyle
  s_{\ell m}(x)
  \Re(\partial^{b}_{\ell}v\overline{\partial^{b}_{m}v})
  \ge
  \frac{n-1}{nR}
  \nu^{2}
  \one{|x|\le R}(x)
  |\nabla^{b}v|^{2}.
\end{equation}
Integrating on $\Omega$ and recalling \eqref{eq:estRD} we obtain
\begin{equation}\label{eq:estInablav}
  \textstyle
  \int_{\Omega}
  I_{|\nabla v|^{2}}
  dx
  \ge
  \frac{n-1}{nR}\nu^{2}
  \int_{\Omega_{\le R}}
  |\nabla^{b}v|^{2}dx
  -
  24NC_{a}\delta^{-1}\|\nabla^{b}v\|^{2}_{Y}.
\end{equation}

\subsection{Estimate of the magnetic terms}
\label{sub:estimate_magnetic}

Consider the term
\begin{equation*}
  I_{b}:=
  2\Im[a_{jk}\partial^{b}_{k}v
     (\partial_{j}b_{\ell}-\partial_{\ell}b_{j})
     a_{\ell m}\partial_{m}\psi\ \overline{v}]
  =
  2\Im 
  \left[
    (db \cdot a \widehat{x}) \cdot (a \nabla^{b}v)
    \overline{v}\psi'
  \right]
\end{equation*}
where the identity holds for any radial $\psi$, while 
$db$ is the matrix 
\begin{equation*}
  db=[\partial_{j}b_{\ell}-\partial_{\ell}b_{j}]_{j,\ell=1}^{n}.
\end{equation*}
Since $0\le \psi'\le 1/2$, $|a(x)|\le N$, we have
\begin{equation*}
  \textstyle
  |I_{b}(x)|\le 
  2N^{2}|db(x)|\cdot|\nabla^{b}v||v|\psi'
  \le
  N^{2}|db(x)|\cdot|\nabla^{b}v||v|
\end{equation*}
and by assumption \eqref{eq:assdb}
we get
\begin{equation}\label{eq:estIb}
  \textstyle
  \int_{\Omega}
  |I_{b}|\le 
   N^{2}C_{b}\int_{\Omega}
   \frac{|\nabla^{b}v||v| }{{|x|^{2+\delta}+|x|^{2-\delta}}}
   \le
   5 \delta^{-1}N^{2}C_{b}
   (\|\nabla^{b}v\|_{\dot Y}^{2}+\|v\|_{\dot X}^{2})
\end{equation}
using \eqref{eq:stnorma10}. Under the stronger
assumption \eqref{eq:assdbstrong}
we have instead, using \eqref{eq:stnorma13}
\begin{equation}\label{eq:estIbbis}
  \textstyle
  \int_{\Omega}
  |I_{b}|\le 
   N^{2}C_{b}\int_{\Omega}
   \frac{|\nabla^{b}v||v| }{{|x|^{2+\delta}+|x|}}
   \le
   9 \delta^{-1}N^{2}C_{b}
   (\|\nabla^{b}v\|_{ Y}^{2}+\|v\|_{ X}^{2}).
\end{equation}

\subsection{Estimate of the terms containing \texorpdfstring{$f$}{}}
\label{sub:estimate_of_the_terms_in_f_}

Consider now the terms
\begin{equation*}%
  I_{f}:=
  \Re (A\psi+\phi)\overline{v}f+
  2\Re a(\nabla\psi,\nabla^bv)f
\end{equation*}
with $\phi \equiv0$.
We have easily
\begin{equation*}
  \textstyle
  \psi'',\frac{\psi'}{|x|}\le \frac{1}{2(R\vee|x|)}
  \quad\implies\quad
  0\le \widehat{a}\psi''+\frac{\overline{a}-\widehat{a}}{|x|}
  \psi'
  \le \frac{\overline{a}}{2(R\vee|x|)}
  \le \frac{Nn}{2(R\vee|x|)}
\end{equation*}
and recalling \eqref{eq:Apsiphi} we get
\begin{equation*}
  \textstyle
  |A \psi|\le
  \frac{Nn}{2(R\vee|x|)}+\frac{1}2 |a'|\le
  \frac{Nn}{2(R\vee|x|)}+
  \frac{C_{a}}{2\langle x\rangle^{1+\delta}}
\end{equation*}
by $\psi'\le1/2$ and assumption \eqref{eq:assader}. Thus by
\eqref{eq:stnorma8} we can write, for all $R>0$,
\begin{equation*}
  \textstyle
  \int_{\Omega}|(A \psi) \overline{v}f|
  \le
  \frac 12(Nn+C_{a})\|f\|_{\dot Y^{*}}\||x|^{-1}v\|_{\dot Y}
  \le
  \frac 12(Nn+C_{a})\|f\|_{\dot Y^{*}}\|v\|_{\dot X}
\end{equation*}
while, using instead the second estimate in \eqref{eq:stnorma8},
we can write for all $R>1$
\begin{equation*}
  \textstyle
  \int_{\Omega}|(A \psi) \overline{v}f|
  \le
  (Nn+C_{a})\|f\|_{Y^{*}}\|\bra{x}^{-1}v\|_{Y}
  \le
  (Nn+C_{a})\|f\|_{Y^{*}}\|v\|_{X}.
\end{equation*}
For the second term in $I_{f}$ we have simply
\begin{equation*}
  \textstyle
  |2a(\nabla \psi,\nabla^{b}v)f|\le
  N|\nabla^{b}v||f|
\end{equation*}
and summing up, for any $0<\zeta_{1}\le1$, we obtain
\begin{equation}\label{eq:estIf}
  \textstyle
  R>0 \quad\implies\quad
  \int_{\Omega}|I_{f}|
  \le
  (Nn+C_{a})^{2}\zeta_{1}^{-1}\|f\|_{\dot Y^{*}}^{2}
  +\zeta_{1}[\|v\|_{\dot X}^{2}+\|\nabla^{b}v\|_{\dot Y}^{2}]
\end{equation}
and also
\begin{equation}\label{eq:estIfbis}
  \textstyle
  R>1 \quad\implies\quad
  \int_{\Omega}|I_{f}|
  \le
  (Nn+C_{a})^{2}\zeta_{1}^{-1}\|f\|_{ Y^{*}}^{2}
  +\zeta_{1}[\|v\|_{ X}^{2}+\|\nabla^{b}v\|_{ Y}^{2}].
\end{equation}

\subsection{Estimate of the boundary terms}
\label{sub:boundary_terms}

It remains to consider the term in divergence form 
$\partial_{j}\Re Q_{j}$ appearing in identity 
\eqref{eq:morid}, where
\begin{equation*}
  Q_{j}=
    a_{jk}\nabla^{b}_{k}v \cdot 
    [A^{b},\psi]\overline{v}
    -\frac12 a_{jk}(\partial_{k}A \psi)|v|^{2}
    -a_{jk}\partial_{k}\psi 
    \left[(c-\lambda)|v|^{2}+a(\nabla^{b}v,\nabla^{b}v)\right].
\end{equation*}
Note that
\begin{equation*}
  [A^{b},\psi]\overline{v}=
  (A \psi)\overline{v}+2a(\nabla \psi,\nabla^{b}v).
\end{equation*}
As before we integrate over the set $\Omega\cap \{|x|\le R\}$
and then let $R\to+\infty$. The integral over $|x|=R$ tends
to zero, and we are left to consider the integral over 
$\partial \Omega$.
After canceling several terms due to the Dirichlet 
boundary condition,
and noticing that $\nabla^{b}v=\nabla v+ibv=\nabla v$ on 
$\partial \Omega$, we are left with
\begin{equation}
  \textstyle
  \int_{\Omega}\partial_{j}\Re Q_{j}
  =
  \Re
  \int_{\partial \Omega}
  \left[
    2a(\nabla v,\vec{\nu})\cdot
    a(\widehat{x},\nabla v)
    -
    a(\nabla v,\nabla v)\cdot
    a(\widehat{x},\vec{\nu})
  \right]\psi'
  dS
\end{equation}
where $\vec{\nu}$ is the exterior unit normal to 
$\partial \Omega$.
Dirichled boundary conditions imply that
$\nabla v$ is normal to $\partial \Omega$ so that
\begin{equation*}
  \nabla v=(\vec{\nu}\cdot \nabla v)\vec{\nu}
\end{equation*}
and hence
\begin{equation*}
  a(\nabla v,\vec{\nu})=(\vec{\nu}\cdot \nabla v)
  a(\vec{\nu},\vec{\nu}),
  \quad
  a(\widehat{x},\nabla v)=
  (\vec{\nu}\cdot \nabla \overline{v})a(\widehat{x},\vec{\nu}),
  \quad
  a(\nabla v,\nabla v)=
  |\vec{\nu}\cdot \nabla \overline{v}|^{2}a(\vec{\nu},\vec{\nu})
\end{equation*}
and
\begin{equation*}%
  \textstyle
  \int_{\Omega}\partial_{j}\Re Q_{j}
  =
  \int_{\partial \Omega}
  |\vec{\nu}\cdot \nabla \overline{v}|^{2}
  a(\vec{\nu},\vec{\nu})
  a(\widehat{x},\vec{\nu})\psi'
  dS.
\end{equation*}
In particular when the obstacle 
$\mathbb{R}^{n}\setminus\Omega$ satisfies
\eqref{eq:assbdry} we have
\begin{equation}\label{eq:estbdry}
  \textstyle
  \int_{\Omega}\partial_{j}\Re Q_{j}
  \le0.
\end{equation}

\subsection{Conclusion of the proof}\label{sub:conclusion}

We now integrate the identity \eqref{eq:morid} on $\Omega$;
by \eqref{eq:estbdry} we obtain
\begin{equation}\label{eq:structest}
  \textstyle
  \int_{\Omega}(I_{|v|^{2}}+I_{|\nabla v|^{2}})\le
  \int_{\Omega}(|I_{\epsilon}|+|I_{b}|+|I_{f}|).
\end{equation}

\subsubsection{Proof of Theorem \ref{the:1}}\label{sub:thm1}
We use estimates 
\eqref{eq:estIep}, \eqref{eq:estIv2},
\eqref{eq:estInablav}, \eqref{eq:estIb},
\eqref{eq:estIf}, and take the sup over $R>0$.
Note in particular that by \eqref{eq:estInablav} we have
\begin{equation*}%
  \textstyle
    \sup_{R>0}\int_{\Omega}I_{|\nabla v|^{2}}dx\ge
    \left[
      \frac{n-1}{n}\nu^{2}-24nNC_{a}\delta^{-1}
    \right]
    \|\nabla ^{b}v\|^{2}_{\dot Y}.
\end{equation*}
Thus we obtain
\begin{equation*}
  M_{1}\|v\|_{\dot X}^{2}+
  M_{2}\|\nabla^{b}v\|_{\dot Y}^{2}
  \le
  M_{3}\|f\|_{\dot Y^{*}}^{2}
\end{equation*}
where
\begin{equation*}
  M_{1}=
  K_{0}\nu^{2}-24nC_{a}(N+C_{a})\delta^{-1}
  -C_{c}\delta^{-1}
  -18 \delta^{-1}N(N+2)C_{-}
  -5 \delta^{-1}N^{2}C_{b}
  -\zeta-\zeta_{1}
\end{equation*}
\begin{equation*}
  \textstyle
  M_{2}=
  \frac{n-1}{n}\nu^{2}-24NC_{a}\delta^{-1}
  -18 \delta^{-1}NC_{-}
  -5 \delta^{-1}N^{2}C_{b}
  -\zeta-\zeta_{1}
\end{equation*}
\begin{equation*}
  M_{3}=
  8(N+2)^{2}(C_{-}^{2}+C_{+}^{2}+Nn+n^{2}C_{a}+1)\zeta^{-1}
  +(Nn+C_{a})^{2}\zeta_{1}^{-1}.
\end{equation*}
Recall that
\begin{equation*}
  \textstyle
  K_{0}:=
    \frac{3n-1}{2}-\frac{n+3}{2}\frac{N}{\nu}>0
    \quad\implies\quad N\le 3\nu
\end{equation*}
while $\zeta,\zeta_{1}\in(0,1]$ are arbitrary.
We define
\begin{equation*}
  \textstyle
  K:=\min\left\{1,\frac{\nu^{2}}{9},\frac{\nu^{2}}{9}K_{0}\right\},
\end{equation*}
we impose the conditions
\begin{equation*}
  24nC_{a}(N+C_{a})\delta^{-1},\quad
  C_{c}\delta^{-1},\quad
  18 \delta^{-1}N(N+2)C_{-},\quad
  5 \delta^{-1}N^{2}C_{b} \quad
  \ \text{are}\  
  \le
  K
\end{equation*}
and we choose $\zeta=\zeta_{1}=K$. Then we get
\begin{equation*}
  \textstyle
  \|v\|_{\dot X}^{2}+\|\nabla^{b}v\|_{\dot Y}^{2}\le
  K^{-1}M_{3}\|f\|_{\dot Y^{*}}^{2}
\end{equation*}
It is easily checked that with our choices
$C_{-}\le\nu$, $C_{a}\le\nu$, $N\le3\nu$
so that
\begin{equation*}
  \textstyle
  M_{3}\le 64n^{2} K^{-1}(\nu+1)^{2} (C_{+}+\nu+1)^{2}
\end{equation*}
and finally
\begin{equation}\label{eq:almost}
  \textstyle
  \|v\|_{\dot X}^{2}+\|\nabla^{b}v\|_{\dot Y}^{2}\le
  64n^{2}K^{-2}
  (\nu+1)^{2} (C_{+}+\nu+1)^{2}\|f\|_{\dot Y^{*}}^{2}
\end{equation}
while the previous conditions on the coefficients
are implied by assumptions \eqref{eq:smallC}.

On the other hand, estimates \eqref{eq:lambdapos},
\eqref{eq:lambdaneg} with the previous choices give
\begin{equation*}
  \lambda_{+}\|v\|_{\dot Y}^{2}\le
  6\nu\|\nabla^{b}v\|_{\dot Y}^{2}
  +4(C_{+}^{2}+3\nu(\nu+1)+n^{2}\nu+1)\|v\|_{\dot X}^{2}
  +\|f\|_{\dot Y^{*}}^{2},
\end{equation*}
\begin{equation*}
  \lambda_{-}\|v\|_{\dot Y}^{2}\le
  2(\nu^{2}+3\nu+n^{2}\nu+1)\|v\|_{\dot X}^{2}
  +\|f\|_{\dot Y^{*}}^{2}
\end{equation*}
while estimate \eqref{eq:epsest} gives
\begin{equation*}
  |\epsilon|\|v\|_{\dot Y}^{2}\le
  3(3\nu+1)
  \left[
    \|v\|_{\dot X}^{2}+\|\nabla^{b}v\|_{\dot Y}^{2}
    +\|f\|_{\dot Y^{*}}^{2}
  \right].
\end{equation*}
Summing up, and using \eqref{eq:almost}, we 
obtain \eqref{eq:thesisA2}.

\subsubsection{Proof of Theorem \ref{the:2}}\label{sub:thm2}
We use estimates 
\eqref{eq:estIepbis}, \eqref{eq:estIv2bis},
\eqref{eq:estInablav}, \eqref{eq:estIbbis},
\eqref{eq:estIfbis} and take the sup over $R>1$.
In particular we have, by \eqref{eq:estInablav},
\begin{equation*}%
  \textstyle
    \sup_{R>1}\int_{\Omega}I_{|\nabla v|^{2}}dx\ge
    \left[
      \frac{n-1}{n}\nu^{2}-24 NC_{a}\delta^{-1}
    \right]
    \|\nabla ^{b}v\|^{2}_{Y}.
\end{equation*}
Note also that we can take $\nu=1-C_{I}$ 
and $N=1+C_{I}$ by assumption
\eqref{eq:assperturb}, while $n=3$. We obtain
\begin{equation*}
  \textstyle
  (1-C_{I})\sup_{R>1}\frac{1}{R^{2}}\int_{\Omega_{=R}}|v|^{2}dS
  -m_{1}\|v\|^{2}_{X}
  +M_{2}\|\nabla^{b}v\|^{2}_{Y}
  \le
  M_{3}\|f\|^{2}_{Y^{*}}
\end{equation*}
where
\begin{equation*}
  m_{1}=
  8 \delta^{-1}[41C_{a}(N+C_{a})+9C_{I}+C_{c}]
  +18 \delta^{-1}N(N+2)C_{-}
  +9 \delta^{-1}N^{2}C_{b}
  +\zeta+\zeta_{1}
\end{equation*}
\begin{equation*}
  \textstyle
  M_{2}=
  \frac23(1-C_{I})^{2}-24 NC_{a}\delta^{-1}
  -13 \delta^{-1}[36 C_{a}(N+C_{a})+C_{c}]
  -18 \delta^{-1}N(N+2)C_{-}
  -9 \delta^{-1}N^{2}C_{b}
  -\zeta-\zeta_{1}
\end{equation*}
\begin{equation*}
  M_{3}=
  44(C_{-}^{2}+4C_{+}^{2}+3N+9C_{a}+1)(N+2)^{2}\zeta^{-1}
  +(3N+C_{a})^{2}\zeta^{-1}_{1}.
\end{equation*}
We impose that $C_{I}\le1/100$ so that we can take
$N=101/100$, $\nu=99/100$.
If we choose $\zeta=\zeta_{1}=1/200$ and
impose that the quantities
\begin{equation*}
  13 \delta^{-1} \cdot 36 C_{a}(N+C_{a}),\quad
  72 \delta^{-1}C_{I},\quad
  13 \delta^{-1}C_{c},\quad
  18 \delta^{-1}N(N+2) C_{-},\quad
  9 \delta^{-1}N^{2}C_{b}
\end{equation*}
are $\le 1/100$, we see that 
$24 NC_{a}\delta^{-1}\le6 \cdot10^{-5}$,
$8 \delta^{-1}\cdot41C_{a}(N+C_{a})\le8 \cdot10^{-3}$,
$C_{a}\le 3 \cdot10^{-5}$,
$C_{-}^{2}\le 2 \cdot 10^{-4}$,
$8 \delta^{-1}C_{c}\le 7 \cdot 10^{-3}$
too. With these choices, we obtain easily
\begin{equation*}
  \textstyle
  \frac{99}{100}
  \sup_{R>1}\frac{1}{R^{2}}\int_{\Omega_{=R}}|v|^{2}dS
  -\frac{46}{1000}\|v\|^{2}_{X}
  +\frac{60}{100}\|\nabla^{b}v\|^{2}_{Y}
  \le
  38000(C_{+}^{2}+1)\|f\|^{2}_{Y^{*}}
\end{equation*}
and using \eqref{eq:stnorma7} we arrive at
\begin{equation}\label{eq:almostbis}
  \|v\|_{X}^{2}+\|\nabla^{b}v\|_{Y}^{2}
  \le
  5 \cdot10^{8}\|f\|_{Y^{*}}^{2}.
\end{equation}
It is straightforward to check that
the previous conditions are (generously)
implied by the smallness assumptions \eqref{eq:smallCstrong}.

We recall also estimates \eqref{eq:lambdaposbis} and
\eqref{eq:lambdanegbis}, which
with the previous choice of constants imply
\begin{equation*}
  \lambda_{+}\|v\|_{Y}^{2}\le
  18(C_{+}^{2}+1)
  (\|\nabla^{b}v\|_{Y}^{2} +\|v\|_{X}^{2})
  +\|f\|_{Y^{*}},
\end{equation*}
\begin{equation*}
  \lambda_{-}\|v\|_{Y}^{2}\le
  10^{-4}\|\nabla^{b} v\|_{Y}^{2}+7\|v\|_{X}^{2}+\|f\|^{2}_{Y^{*}},
\end{equation*}
while estimate \eqref{eq:epsestbis} implies now
\begin{equation*}
  |\epsilon|\|v\|_{Y}^{2}\le
  5\|v\|_{X}^{2}+5\|\nabla^{b}v\|_{Y}^{2}+\|f\|_{Y^{*}}^{2}.
\end{equation*}
Taking into account \eqref{eq:almostbis}, we finally obtain
\eqref{eq:thesisB2}.

\section{Proof of the Corollaries}\label{sec:proof_coroll}

The proof of Corollary \ref{cor:smooschr} is standard.
Denote by $R(z)=(L-z)^{-1}$ the resolvent operator of $L$,
then estimates \eqref{eq:thesisA} and \eqref{eq:thesisB}
imply
\begin{equation*}
  \|\bra{x}^{-\frac32-} R(z)v\|_{L^{2}(\Omega)}+
  \|\bra{x}^{-\frac12-}\nabla R(z)v\|_{L^{2}(\Omega)}
  \lesssim
  \|\bra{x}^{\frac12+}v\|_{L^{2}(\Omega)}
\end{equation*}
where we used estimates \eqref{eq:stnorma1}, \eqref{eq:stnorma3}
and \eqref{eq:stnorma4} from Section \ref{sec:norms}.
From the first term by complex interpolation we obtain
\begin{equation*}
  \|\bra{x}^{-1-} R(z)\bra{x}^{-1-}v\|_{L^{2}(\Omega)}
  \lesssim
  \|v\|_{L^{2}(\Omega)}
\end{equation*}
which means that the operator $\bra{x}^{1+}$ is
$L$-supersmooth; on the other term, from the second term
and the elliptic estimate \eqref{eq:ellest} we have
\begin{equation*}
\begin{split}
  \|\bra{x}^{-\frac12-}L^{\frac14}
    R(z)L^{\frac14}v\| _{L^{2}(\Omega)}
  =
  \|\bra{x}^{-\frac12-}L^{\frac12}
    R(z)v\| _{L^{2}(\Omega)}
  &\lesssim
  \|\bra{x}^{-\frac12-}\nabla
    R(z)v\| _{L^{2}(\Omega)}
  \\
  &\lesssim
  \|\bra{x}^{\frac12+}v\|_{L^{2}(\Omega)}
\end{split}
\end{equation*}
and this implies that $\bra{x}^{-\frac12-}L^{\frac14}$
is also $L$-supersmooth. From the standard Kato theory
(see the details in \cite{DAncona14-a}) we obtain
\eqref{eq:smooschr}, \eqref{eq:smooschrnh1}, 
\eqref{eq:smooschrnh2}.

Also the proof of \eqref{eq:ragesch} is standard: we can
write for any $f\in D(L)$ and any test function $g$
(scalar product and norm of $L^{2}(\Omega)$)
\begin{equation*}
  \textstyle
  (e^{i(t+1)L} f,g)=
  (\int_{t}^{t+1}\partial_{s}[(s-t)e^{isL}f]ds,g)=
  i\int_{t}^{t+1}((s-t)Le^{isL}f,g)ds
     -\int_{t}^{t+1}(e^{isL}f,g)ds
\end{equation*}
which implies by Cauchy-Schwartz
\begin{equation*}
  \textstyle
  |(e^{i(t+1)L} f,g)|\le
  \|\bra{x}^{\frac12+}g\|
  \left(\int_{t}^{\infty}
      \|\bra{x}^{-\frac12-}e^{itL}Lf\|^{2}
      +\|\bra{x}^{-\frac12-}e^{itL}f\|^{2}
      ds\right)^{\frac12}.
\end{equation*}
Using \eqref{eq:smooschr} we obtain
\begin{equation*}
  \textstyle
  \int_{0}^{\infty}
  \|\bra{x}^{-\frac12-}e^{itL}Lf\|^{2}ds
  \le
  \|L^{\frac34}f\|^{2}
\end{equation*}
and this implies that for every $\epsilon>0$ we can find
$T_{\epsilon}$ such that 
\begin{equation*}
  \textstyle
  \int_{t}^{\infty}
  \|\bra{x}^{-\frac12-}e^{itL}Lf\|^{2}ds
  \le
  \epsilon
  \quad\text{for}\quad t>T_{\epsilon}.
\end{equation*}
The second term can be handled in a similar way. Thus we have
proved that for any $f\in D(L)$, any test function $g$
and any $\epsilon>0$
\begin{equation*}
  |(e^{itL} f,g)|\le
  \|\bra{x}^{\frac12+}g\| \cdot \epsilon
  \quad\text{for}\quad t>T_{\epsilon}+1
\end{equation*}
which implies
\begin{equation*}
  \|\bra{x}^{-\frac12-}e^{itL}f\|\le\epsilon
  \quad\text{for}\quad t>T_{\epsilon}+1.
\end{equation*}
By density (note that $L$ has no eigenvalues) we 
obtain \eqref{eq:ragesch}.

We recall if $L$ is a selfadjoint
operator on a Hilbert space $\mathcal{H}$ and
$A:\mathcal{H}\to \mathcal{H}_{1}$
is a closed operator to a second Hilbert space
$\mathcal{H}_{1}$, $A$ is called
$L$-\emph{supersmooth} if the following estimate
holds for all $z\in \mathbb{C}\setminus \mathbb{R}$
\begin{equation*}
  \|A(z-L)^{-1}A^{*}v\|_{\mathcal{H}_{1}}\le C\|v\|_{\mathcal{H}}
\end{equation*}
with a constant uniform in $z$. From this estimate one obtains
the following estimates for the Schr\"{o}dinger flow
associated to $L$:
\begin{equation*}
  \|Ae^{-itL}v\|_{L^{2}_{t}\mathcal{H}_{1}}\le
  C'\|v\|_{\mathcal{H}},
  \qquad
  \textstyle
  \|\int_{0}^{t}Ae^{i(t-s)L}A^{*}h(s)ds\|
      _{L^{2}_{t}\mathcal{H}_{1}}\le
  C'\|h\|_{L^{2}_{t}\mathcal{H}_{1}}.
\end{equation*}
The proof of Corollary \ref{cor:smoowe} is immediate
using the following result
which is a special case of Corollary 2.5
in \cite{DAncona14-a}:

\begin{proposition}
  Let $L\ge0$ be a selfadjoint operator on the
  Hilbert space $\mathcal{H}$ and let $P$
  be the orthogonal projection onto $\ker(L)^{\perp}$.
  Assume $A$ and $AL^{-\frac14}P$ are closed operators with
  dense domain from $\mathcal{H}$ to a second
  Hilbert space $\mathcal{H}_{1}$.
  If $A$ is $L$-supersmooth then
  $AL^{-\frac14}P$ is $\sqrt{L}$-supersmooth.
\end{proposition}

Applying this result we obtain
\begin{equation*}
  \text{$\bra{x}^{1+}$ is $L$-supersmooth}
  \quad\implies\quad 
  \text{$\bra{x}^{1+}L^{-\frac14}$ is $\sqrt{L}$-supersmooth}
\end{equation*}
and
$\bra{x}^{-\frac12-}L^{\frac14}$
\begin{equation*}
  \text{$\bra{x}^{-\frac12-}L^{\frac14}$ is $L$-supersmooth}
  \quad\implies\quad 
  \text{$\bra{x}^{-\frac12-}$ is $\sqrt{L}$-supersmooth}
\end{equation*}
and Corollary \ref{cor:smoowe} follows directly. The
RAGE-type property is proved similarly to the previous one.

\bibliographystyle{plain}

\end{document}